\documentclass[a4paper, 10pt]{amsart}
\usepackage[margin=1.2in, marginparwidth=0.7cm]{geometry}
\usepackage{amssymb, amsthm, amsfonts}
\usepackage[centertags]{amsmath}
\usepackage{eucal}
\usepackage{mathrsfs}
\usepackage{marginnote}
\usepackage[all,cmtip]{xy}
\usepackage{graphicx, eepic}
\usepackage[shortlabels]{enumitem} 
\usepackage{xcolor}
\usepackage{bbm}
\usepackage{pdfpages}
\usepackage{titletoc}
\usepackage{booktabs}
\usepackage{mathtools}
\usepackage{thmtools}
\usepackage{multirow, url}
\usepackage{textcomp}
\DeclareMathAlphabet{\cmcal}{OMS}{cmsy}{m}{n}

\usepackage{fancyhdr} 
\usepackage{lmodern}
\usepackage[utf8]{inputenc}
\usepackage[T1]{fontenc}

\definecolor{blue1}{rgb}{0, 0, 2}
\definecolor{sky}{rgb}{0, 0.2, 0.8}
\usepackage[colorlinks=true, citecolor=sky, linkcolor=magenta, urlcolor=blue1, filecolor=cyan, backref=page
]{hyperref}
\newtheorem{theorem}{Theorem}[section]
\newtheorem{corollary}[theorem]{Corollary}
\newtheorem{lemma}[theorem]{Lemma}
\newtheorem{proposition}[theorem]{Proposition}

\newtheorem{conjecture}[theorem]{Conjecture}

\theoremstyle{definition}
\newtheorem{definition}[theorem]{Definition}

\theoremstyle{remark}
\newtheorem{remark}[theorem]{\bf{Remark}}
\newtheorem{assumption}[theorem]{\bf{Assumption}}

\newtheorem{notation}[theorem]{\bf{Notation}}

\newtheorem{question}[theorem]{\bf{Question}}

\newtheorem{example}[theorem]{\bf{Example}}

\newtheorem*{theorem*}{\bf{Theorem}}
\newtheorem*{claim*}{\bf{Claim}}
\newtheorem*{remark*}{\bf{Remark}}
\newtheorem*{remarks*}{\bf{Remarks}}
\newtheorem*{example*}{\bf{Example}}
\newtheorem*{examples*}{\bf{Examples}}

\newcommand{\C}{{\mathbb{C}}}

\newcommand{\F}{{\mathbb{F}}}

\newcommand{\bP}{{\mathbb{P}}}
\newcommand{\Q}{{\mathbb{Q}}}

\newcommand{\T}{{\mathbb{T}}}

\newcommand{\Z}{{\mathbb{Z}}}

\newcommand{\fa}{{\mathfrak{a}}}
\newcommand{\fb}{{\mathfrak{b}}}

\newcommand{\fm}{{\mathfrak{m}}}
\newcommand{\fn}{{\mathfrak{n}}}

\newcommand{\cC}{{\cmcal{C}}}

\newcommand{\cE}{{\cmcal{E}}}
\newcommand{\cF}{{\cmcal{F}}}

\newcommand{\cI}{{\cmcal{I}}}
\newcommand{\cJ}{{\cmcal{J}}}

\newcommand{\cQ}{{\cmcal{Q}}}

\newcommand{\cS}{{\cmcal{S}}}
\newcommand{\cT}{{\cmcal{T}}}

\def\a{\alpha}
\def\b{\beta}

\def\d{\delta}
\def\e{\epsilon}

\def\g{\gamma}

\def\l{\lambda}
\def\o{\omega}
\def\s{\sigma}
\def\z{\zeta}
\def\p{\varphi}
\def\vp{\varpi}

\def\D{\Delta}
\def\G{\Gamma}

\def\S{\Sigma}
\def\O{\Omega}

\def\<{\left\langle}
\def\>{\right\rangle}

\newcommand{\zmod}[1]{{\Z/{#1}\Z}}

\newcommand{\inj}{\hookrightarrow}
\newcommand{\surj}{\twoheadrightarrow}
\newcommand{\arinj}{\ar@{^(->}}
\newcommand{\arsurj}{\ar@{->>}}
\newcommand{\arsub}{\ar@{}[r]|-*[@]{\subset}}
\newcommand{\arsup}{\ar@{}[r]|-*[@]{\supset}}
\newcommand{\arcap}{\ar@{}[d]|-*[@]{\subset}}
\newcommand{\arcup}{\ar@{}[u]|-*[@]{\subset}}
\newcommand{\arin}{\ar@{}[u]|-*[@]{\in}}

\renewcommand{\pmod}[1]{{\,(\operatorname{mod}\hspace{0.7mm} {#1})}}
\renewcommand{\mod}[1]{{\,\operatorname{mod}\hspace{0.7mm} {#1}}}

\newcommand{\Ext}{{\operatorname{Ext}}}
\newcommand{\Hom}{{\operatorname{Hom}}}
\newcommand{\Gal}{{\operatorname{Gal}}}

\newcommand{\Pic}{{\operatorname{Pic}}}

\newcommand{\End}{{\operatorname{End}}}

\newcommand{\GL}{{\operatorname{GL}}}

\newcommand{\Ta}{{\operatorname{Ta}}}
\newcommand{\Spec}{{\operatorname{Spec}}}

\newcommand{\sHom}{{\mathscr{H}\kern-.5pt om}}
\newcommand{\sExt}{{\mathscr{E}\kern-.5pt xt}}

\newcommand{\Frob}{{\mathsf{Frob}}}

\newcommand{\old}{{\mathsf{old}}}
\newcommand{\new}{{\mathsf{new}}}

\newcommand{\sqf}{{\mathsf{sqf}}}

\renewcommand{~}{\hspace*{0.5mm}}

\newcommand{\ts}{{\textsection}}
\newcommand{\ra}{\rightarrow}
\newcommand{\wt}{\widetilde}

\newcommand{\ov}{\overline}
\newcommand{\ud}{\underline}
\newcommand{\bd}{\boldsymbol}

\newcommand{\ms}{\medskip}
\newcommand{\sm}{\smallsetminus}

\mathchardef\hyp="2D
\newcommand{\xyv}[1]{\xymatrixrowsep{#1 pc}}
\newcommand{\xyh}[1]{\xymatrixcolsep{#1 pc}}

\newcommand{\qa}{{\quad \text{and} \quad}}

\newcommand{\mat}[4]{
 \left(  \begin{smallmatrix} #1 & #2 \\ #3 & #4 \end{smallmatrix} \right)}
 
\newcommand{\bmat}[4]{
  \left( \begin{array}{cc} #1 & #2 \\ #3 & #4 \end{array} \right)}

\newcommand{\vect}[2]{
 \left(  \begin{smallmatrix} #1 \\ #2 \end{smallmatrix} \right)}

\newcommand{\bect}[2]{
  \left[ \begin{smallmatrix} #1  \\ #2 \end{smallmatrix} \right]}

\newcommand{\plim}[1]{\lim\limits_{\substack{\longleftarrow\\{#1}}}\;}
\newcommand{\plims}[1]{\lim\limits_{\substack{\leftarrow ~{#1}}}\;}
\newcommand{\ilim}[1]{\lim\limits_{\substack{\longrightarrow\\{#1}}}\;}
\newcommand{\ilims}[1]{\lim\limits_{\substack{\rightarrow ~{#1}}}\;}

\begin{document}

\title{The kernel of a rational Eisenstein prime at non-squarefree level}

\author{Hwajong Yoo}
\address{IBS Center for Geometry and Physics,
Mathematical Science Building, Room 108,
Pohang University of Science and Technology,
77 Cheongam-ro, Nam-gu, Pohang, Gyeongbuk, Republic of Korea 37673}
\email{hwajong@gmail.com}                                                   
\thanks{This work was supported by IBS-R003-D1.}                 
\subjclass[2010]{11F33, 11F80 (Primary); 11G18 (Secondary)}    
\maketitle

\begin{abstract}
Let $\ell \geq 5$ be a prime and let $N$ be a non-squarefree integer not divisible by $\ell$. For a rational Eisenstein prime $\fm$ of the Hecke ring $\T(N)$ of level $N$ acting on $J_0(N)$, we precisely compute the dimension of the kernel $J_0(N)[\fm]$ under a mild assumption. 
In the case of level $qr^2$ which violates our mild assumption, we propose a conjecture based on Sage computations. Assuming this conjecture, we complete our computation in all the remaining cases.
\end{abstract}


\section{Introduction}
Let $N$ be a positive integer. Consider the modular curve $X_0(N)$ over $\Q$ and its Jacobian variety $J_0(N):=\Pic^0(X_0(N))$ over $\Q$. Let $T_n$ denote the $n$-th Hecke operator in the endomorphism ring $\End(J_0(N))$ of $J_0(N)$ and let $\T(N)$ denote the $\Z$-subring of $\End(J_0(N))$ generated by $T_n$ for all integers $n\geq 1$. (For the precise definition, and for details of what follows, see Section \ref{section : background}.)
\ms

For a maximal ideal $\fm$ of $\T(N)$ containing a prime $\ell$, there is a continuous odd semisimple mod $\ell$ Galois representation 
\begin{equation*}
\rho_\fm : G_\Q \to \GL(2, k_\fm),
\end{equation*}
where $G_\Q:=\Gal(\ov{\Q}/\Q)$ and $k_\fm:=T(N)/\fm$, satisfying
\begin{equation*}
\text{trace} (\Frob_p ) = T_p \pmod \fm \qa \text{det} (\Frob_p) = p \pmod \fm
\end{equation*}
for all primes $p$ not dividing $N\ell$ (\cite[Th. 6.7]{DS74} or \cite[Prop. 5.1]{R90}). 
Here $\Frob_p$ denotes a Frobenius element for $p$ in $G_\Q$. The representation $\rho_\fm$ is unique up to isomorphism.

\ms
We say that $\fm$ is {\sf Eisenstein} (resp. {\sf non-Eisenstein}) if $\rho_\fm$ is reducible (resp. irreducible), and 
that $\fm$ is a {\sf rational Eisenstein prime} if $\rho_\fm \simeq \mathbf{1} \oplus \chi_\ell$, where 
$\mathbf{1}$ is the trivial character and $\chi_\ell$ is the mod $\ell$ cyclotomic character. 
Note that a rational Eisenstein prime contains $\cI_0(N)$, where
\begin{equation*}
\cI_0(N):=(T_p-p-1 : \text{ for all primes $p$ not dividing $N$}) 
\end{equation*}
is an Eisenstein ideal of $\T(N)$ (cf. \cite[\ts 3]{Yoo6}). Conversely, if a maximal ideal contains $\ell$ and $\cI_0(N)$, then it is a rational Eisenstein prime by a standard argument using the Brauer-Nesbitt theorem and Chebotarev density theorem. We also note that 
\begin{equation*}
\fm \text{ is an Eisenstein prime} \iff \fm \text{ is a rational Eisenstein prime}
\end{equation*}
when $N$ is squarefree \cite[Prop. 2.1]{Yoo2}, but there do exist Eisenstein primes which are not rational Eisenstein primes if $N$ is not squarefree (Remark \ref{remark : non rational-Eisenstein}). 
\ms

For a maximal ideal $\fm$ of $\T(N)$ containing a prime $\ell$, by the {\sf kernel of $\fm$} we mean the following $k_\fm[G_\Q]$-module:
\begin{equation*}
J_0(N)[\fm]:=\{ x \in J_0(N)(\ov{\Q}) : Tx =0 \text{ for any } T \in \fm \}. 
\end{equation*}
This kernel has been studied for decades. It is known that $\dim_{k_\fm} J_0(N)[\fm] \geq 2$ (Lemma \ref{lemma : dim at least 2}) and 
\begin{equation*}
\dim_{k_\fm} J_0(N)[\fm] = 2 \iff \T_{\fm} \text{ is Gorenstein}
\end{equation*}
where $\T_\fm$ is the completion of $\T(N)$ at $\fm$, i.e., $\T_\fm := \plims k \T(N)/{\fm^k}$ (\cite[Appendix]{Ti97} or \cite[5.16]{Ces}). Note that the Gorenstein property of $\T_\fm$ (for non-Eisenstein $\fm$) plays a crucial role in the proof of Fermat's last theorem by Wiles \cite{Wi95} and Taylor--Wiles \cite{TW95}.
\ms

Suppose that $\fm$ is non-Eisenstein. 
Then by Mazur \cite{M77}, the semisimplification of $J_0(N)[\fm]$, which we denote by $J_0(N)[\fm]^\text{ss}$, is the direct sum of copies of $V_\fm$, where $V_\fm$ is the underlying irreducible two dimensional $k_\fm[G_\Q]$-module for $\rho_\fm$. If $V_\fm$ is absolutely irreducible (which is always true when $\ell\geq 3$), by Boston-Lenstra-Ribet \cite{BLR91}, $J_0(N)[\fm] \simeq V_\fm^{\oplus n}$ for some $n\geq 1$. In most cases, the multiplicity $n$ is one\footnote{If $N$ is a prime, it is true by Mazur \cite{M77} unless $\ell=2$ and $\fm$ is ordinary, i.e., $T_\ell \not\in \fm$. In general, if $\ell$ does not divide $2N$, it follows from \cite[Th. 5.2]{R90}. For more general results, see \cite[Th. 2.1]{Wi95} and the references therein.} and hence $\dim J_0(N)[\fm]=2$ and $J_0(N)[\fm] \simeq V_\fm$. For the cases where $\dim J_0(N)[\fm]$ is greater than $2$, often called {\sf failure of multiplicity one}\footnote{In the context of Jacobians of Shimura curves, see \cite{R90a}.}, see the work of Kilford \cite{Kil02} and of Mazur and Ribet \cite[Th. 2]{MR91}. 
\ms

Now suppose that $\fm$ is Eisenstein. Then $J_0(N)[\fm]$ is not a direct sum of copies of $V_\fm$. 
Instead, $J_0(N)[\fm]^\text{ss}$ is the direct sum of one dimensional $k_\fm[G_\Q]$-modules (cf. \cite[Ch. II, Prop.  14.1]{M77}).
Suppose further that $\fm$ is rational Eisenstein. 
Applying Mazur's argument in \cite{M77}, if $\ell$ does not divide $2N$ then we have an exact sequence
\begin{equation*}
\xymatrix{
0 \ar[r] & \zmod \ell \ar[r] & J_0(N)[\fm] \ar[r] & \mu_\ell^{\oplus n} \ar[r] & 0
}
\end{equation*}
for some $n\geq 1$ (Proposition \ref{proposition : mazur argument}). This ``multiplicity'' $n$ is one if and only if $\dim J_0(N)[\fm]=2$. 
When $N$ is a prime, then Mazur proves $\dim J_0(N)[\fm]=2$ without any assumption on residual characteristic $\ell$  \cite{M77}. His work is generalized to squarefree $N$ by Ribet and the author when $\ell$ does not divide $6N$ under a mild assumption \cite{RY, Yoo1}. In contrast to the previous discussion, the dimension of $J_0(N)[\fm]$ is often greater than $2$. For instance, let $p\equiv q \equiv 1 \pmod \ell$ be two distinct primes. Then, 
$\fm=(\ell, T_p-1, T_q-1, \cI_0(pq))$ is an Eisenstein prime and $\dim J_0(pq)[\fm]=5$.
\ms

In 2016, Lecouturier asked the author the following ``multiplicity one'' question:
\begin{question}\label{question1}
As before, let $\ell\geq 5$ be a prime and $p\equiv 1 \pmod \ell$ a prime. Let 
\begin{equation*}
\fm=(\ell, T_2, T_p-1, \cI_0(4p))
\end{equation*}
be a rational Eisenstein prime of $\T(4p)$. Is $\dim J_0(4p)[\fm]=2$?
\end{question}
In this paper, we answer this question more generally, we are interested in the kernel of a rational Eisenstein prime when $N$ is not squarefree, i.e., there is a prime $p$ such that $p^2$ divides $N$. To introduce our result, we need some notation: Fix a prime $\ell\geq 5$ and let $N$ be a non-squarefree integer prime to $\ell$.  Suppose that $\fm$ is a rational Eisenstein prime which is {\sf new} (see Section \ref{section : background} for the definition). 
Then by \cite[Lemma 2.2]{Yoo6} or \cite[\ts 2.(c)]{M78}, we get
\begin{equation*}
\fm=(\ell, T_p-\e(p), \cI_0(N) : \text{ for all prime divisors $p$ of $N$}), 
\end{equation*}
where $\e(p)=0$ if $p^2$ divides $N$ and $\e(p)=\pm 1$ if $p$ {\sf exactly} divides $N$, which we mean that $p$ divides $N$ but $p^2$ does not divide $N$. 
Therefore by appropriately ordering (and indexing) the prime divisors of $N$, we may assume that
\begin{itemize}
\item
$N=\prod_{i=1}^s p_i \prod_{j=1}^t q_j \prod_{k=1}^u r_k^{e(k)}$ with $e(k)\geq 2$; 
\item
$\fm = \fm_\ell(s, t, u)$, where we define
\begin{equation}\label{equation : m(s, t, u)}
\fm_\ell(s, t, u):=(\ell, T_{p_i}-1, T_{q_j}+1, T_{r_k}, \cI_0(N) : \text{ for all } i, j \text{ and } k).
\end{equation}
\end{itemize}

\begin{assumption}  \label{assumption : m=m(s, t, u)} When we write $\fm:=\fm_\ell(s, t, u)$, we {\sf always assume} that
\begin{itemize}
\item
$p_i \equiv 1 \pmod \ell$ for $1 \leq i \leq s_0$, $p_i \not\equiv 1 \pmod \ell$ for $s_0+1 \leq i \leq s$,
\item
$q_j \equiv -1 \pmod \ell$ for all $1\leq j \leq t$,
\item
$r_k \equiv 1 \pmod \ell$ for $1 \leq k \leq u_0$, $r_k \equiv -1 \pmod \ell$ for $u_0+1 \leq k \leq u_1$, 
\item
$r_k \not\equiv \pm 1 \pmod \ell$ for $u_1+1 \leq k \leq u$.
\end{itemize}
\end{assumption}

Before proceeding, we note that by Lemma \ref{lemma : maximal ideal assumption}
\begin{equation*}
\fm:=\fm_\ell (s, t, u) \text{ is maximal }  \iff s+u \geq 1 \text{ and } s_0+t+u_1 \geq 1.
\end{equation*}
We also note that our convention is not limited at all. 
Assume that $\fm:=\fm_{\ell} (s, t, u)$ is maximal, or equivalently, assume that $s+u \geq 1$ and $s_0+t+u_1 \geq 1$.
If $\fm$ is new then it is $q_j$-new and hence $q_j \equiv -1 \pmod \ell$ (cf. \cite[Th. 1.2(3)]{Yoo2}).
So, the second condition is necessary for $\fm$ to be new. 
In the cases where Assumption \ref{assumption : m=m(s, t, u)} is not fulfilled, we can compute the dimension of $J_0(N)[\fm]$ by computing the dimension for a lower level where Assumption \ref{assumption : m=m(s, t, u)} is satisfied.
(In general, Assumption \ref{assumption : m=m(s, t, u)} does not imply that $\fm$ is new. In other words, it is just a necessary condition for $\fm$ to be new but not a sufficient one\footnote{For example, take $\ell=5$, $p_1=11$ and $p_2=3$. Then $\fm:=\fm_5(2, 0, 0)$ is maximal but $\fm$ is not $p_2$-new.}. 
Instead of assuming $\fm$ is new in which case we don't know what is a ``verifiable'' sufficient condition in general, 
our convention just rules out some cases that cannot be new.)
\ms

Now, we can state our main theorem:
\begin{theorem}\label{theorem : the main theorem}
Let $\ell\geq 5$ be a prime, and let $N=\prod_{i=1}^s p_i \prod_{j=1}^t q_j \prod_{k=1}^u r_k^{e(k)}$ with $e(k)\geq 2$ not divisible by $\ell$.
Assume that $N$ is non-squarefree, i.e., $u\geq 1$.
Let $\fm:=\fm_\ell(s, t, u)$ be a rational Eisenstein prime of $\T(N)$. 
If $s_0 \neq s$, or $u_0 \neq u$, or $t=0$, then we have
\begin{equation*}
\dim J_0(N)[\fm] = 1+s_0+t+u_1.
\end{equation*}
\end{theorem}
The theorem above gives an affirmative answer to Question \ref{question1} because it is the situation with 
\begin{itemize}
\item
$p_1=p$ and $r_1=2$;
\item
$s_0=s=1$, $t=0$, $u_0=u_1=0$ and $u=1$.
\end{itemize}
\ms

The sketch of the proof of this theorem is as follows: 
We first generalize a result of Mazur (Proposition \ref{proposition : mazur argument}).
Then, by the results of Ling-Oesterl\'e and Vatsal on the Shimura subgroup, we can show that $J_0(N)[\fm]$ is a {\sf prosaic nugget} in the sense of Brumer-Kramer (Remark \ref{remark : prosaic nugget}). So we can obtain an upper bound on its dimension (Theorem \ref{theorem : upper bound of dimension}). This bound turns out to be optimal if we find enough submodules which are ramified precisely at the primes in the set $\cS_\fm$ associated to $\fm$ (Theorem \ref{theorem : lower bound of dimension}). Under our assumption that $s_0\neq s$ or $u_0 \neq u$ or $t=0$, we can find such submodules by a level raising method (Proposition \ref{proposition : level raising}). 
\ms

We believe that the assumptions in the theorem above (except that $N$ is non-squarefree) are superfluous and the statement should be true for all the remaining cases. 
In the case of level $qr^2$ which violates the assumptions Sage computations  \cite{Sage} ``support'' our belief, so we propose the following:
\begin{conjecture}\label{conjecture : t=1, u=1}
Let $\ell\geq 5$ be a prime and 
let $N=qr^2$ with $r\equiv -q\equiv 1 \pmod \ell$.
Let 
\begin{equation*}
\fm:=\fm_\ell(0, 1, 1)=(\ell, T_q+1, T_r, \cI_0(qr^2))
\end{equation*}
be a rational Eisenstein prime of $\T(qr^2)$. 
Then, we have
\begin{equation*}
\dim J_0(qr^2)[\fm]=3.
\end{equation*}
\end{conjecture}
This conjecture is true under a certain assumption as follows:
\begin{theorem}\label{theorem : level qr - level qr2}
Let $\ell\geq 5$ be a prime and 
let $N=qr^2$ with $r\equiv -q\equiv 1 \pmod \ell$.
Let 
\begin{equation*}
 \fn=(\ell, T_q+1, T_r-1, \cI_0(qr))
 \end{equation*}
be a rational Eisenstein prime of $\T(qr)$. 
If $\dim_{\F_\ell} J_0(qr)[\fn]=3$ then Conjecture \ref{conjecture : t=1, u=1} holds.
\end{theorem}
By \cite[Th. 4.5]{Yoo1}, $\dim_{\F_\ell} J_0(qr)[\fn]$ is either $2$ or $3$. Calegari and Ribet \cite{RY}
conjecture that 
\begin{equation*}
\dim_{\F_\ell} J_0(qr)[\fn]=3 \iff \text{$r$ splits completely in $\cF_q$}
\end{equation*}
where $\cF_q$ is a certain extension field of $\Q(\mu_\ell)$ of degree $\ell$, which is unramified outside the primes dividing $q$ (see Notation \ref{notation : extension at p}). Therefore, for a fixed prime $q$ with $q\equiv -1 \pmod \ell$ we expect that 
\begin{equation*}
\lim_{r \to \infty} \frac{\# \{ r \in S_\ell: \dim J_0(qr)[\fn]=3 \}}{\#S_\ell} = \frac{1}{\ell}
\end{equation*}
where $S_\ell$ is the set of primes which are congruent to $1$ modulo $\ell$. This discussion ``conjecturally'' explains that our assumption in Theorem \ref{theorem : level qr - level qr2} is often fulfilled but is a somewhat strong constraint.

\begin{example}
Here is ``verification'' of the conjecture of Calegari and Ribet by Sage \cite{Sage}.
\begin{itemize}
\item
For $q=19$ and $\ell=5$, $\dim_{\F_5} J_0(qr)[\fn]=3$ if and only if
\begin{equation*}
r=41, 101, 251, 521, 631, 691, 881 \text{ and } 991 \text{ for } r < 1000.
\end{equation*}
Such primes $r$ are precisely the splitting primes in the extension field $\Q(E[5])$ defined by the $5$-torsion points of $E$, where $E$ is the elliptic curve `38b1' in the Cremona table.
\item
For $q=29$ and $\ell=5$, $\dim_{\F_5} J_0(qr)[\fn]=3$ if and only if
\begin{equation*}
r=181, 191, 251, 401, 491, 541, 601, 701, 911 \text{ and } 971 \text{ for } r < 1000,
\end{equation*}
which are precisely the splitting primes in the extension field $\Q(F[5])$, where $F$ is the elliptic curve `58b1' in the Cremona table.
\end{itemize}
\end{example}
\ms

Without the assumption above, we have not been able to prove Conjecture \ref{conjecture : t=1, u=1}. However, 
we can compute the dimension of $J_0(N)[\fm]$ in all the remaining cases if Conjecture \ref{conjecture : t=1, u=1} holds.
\begin{theorem}\label{theorem : main theorem 2}
Let $\ell\geq 5$ be a prime, and let $N=\prod_{i=1}^s p_i \prod_{j=1}^t q_j \prod_{k=1}^u r_k^{e(k)}$ with $e(k)\geq 2$ not divisible by $\ell$.
Let $\fm:=\fm_\ell(s, t, u)$ be a rational Eisenstein prime of $\T(N)$. 
Suppose that $N$ is not squarefree, i.e., $u\geq 1$.
Suppose further that $t\geq 1$ and
\begin{equation*}
p_i \equiv -q_j \equiv r_k \equiv 1 \pmod \ell \text{ for all } i, j \text{ and } k.
\end{equation*}
In other words, $s_0=s$, $u_0=u$ and $t \geq 1$.
If Conjecture \ref{conjecture : t=1, u=1} holds,
then
\begin{equation*}
\dim J_0(N)[\fm]=1+s+t+u=1+s_0+t+u_0.
\end{equation*}
\end{theorem}
The proof of this theorem uses the same strategy as the proof of Theorem \ref{theorem : the main theorem}. In fact, Conjecture \ref{conjecture : t=1, u=1} holds if and only if $J_0(qr^2)[\fm]$ is precisely ramified at both $q$ and $r$. (At the moment, the author only knows that it is ramified at $r$. The assumption in Theorem \ref{theorem : level qr - level qr2} implies that $J_0(qr^2)[\fm]$ is ramified at $q$ as well.) So, it contains ``enough'' submodules (cf. Remark \ref{remark : when dim=1+Sm}) and hence by a level raising method, we can verify the assumption in Theorem \ref{theorem : lower bound of dimension} as above.
\ms

\begin{remark}
One is immediately led to ask whether our methods can be generalized to general Eisenstein primes. It is clear that which information would be necessary for applying our methods (e.g. generalizations or modifications of the results in Section \ref{section : the structure of the kernel}).
For instance, if $\fm$ is Eisenstein but not rational Eisenstein, then $\rho_\fm \simeq \a \oplus \a^{-1} \chi_\ell$ for some non-trivial character $\a : G_\Q \to \ov{\F}_\ell^\times$ \cite[Th. 2.3]{FJ95}. To compute $\dim J_0(N)[\fm]$ for such $\fm$, some ideas in \cite{Wi80} are useful. We will not discuss this generalization in detail. 
\end{remark}
\begin{remark} \label{remark : non rational-Eisenstein}
An example of an Eisenstein prime which is not a rational Eisenstein prime (in level $\ell^2$, which does not fit into our framework) can be constructed as follows: Let $\o$ be the Teichm\"uller character, which is congruent to $\chi_\ell$ modulo $\ell$. 
Then take an Eisenstein series $E$ on $\G_1(\ell)$ of weight $2$ and character $\e$. 
The character $\e$ is ($\ell$-adically) of the form $\o^{k-2}$ for some even $k$.
Then, the constant term of $E$ is congruent to $\l B_k$, where $\l$ is a $\ell$-adic unit and $B_k$ is the $k$-th Bernoulli number. Suppose $B_k$ is divisible by $\ell$ (for some irregular prime $\ell$). Then, there is a cusp form $f$ of weight $2$ and character $\e$ which is congruent to $E$ modulo $\ell$ by Ribet \cite{R76}. The mod $\ell$ representation associated to $f$ is $\mathbf{1} \oplus\chi_\ell \e$.
So, the twist $\wt{f}$ of $f$ by $\o^{(2-k)/2}$ is a cusp form for $\G_0(\ell^2)$ whose mod $\ell$ Galois representation is $\a \oplus \a^{-1} \chi_\ell$ as desired. See \cite[\ts 5]{M78}, \cite[Th. 2.1]{Wi80} or \cite[Ex. 4.10]{Ste85} for relevant discussions on the character $\e$. See also \cite{GL86} for the discussion on Eisenstein primes of level $\ell^2$.
\end{remark}
\ms

This paper will proceed as follows. In Section \ref{section : background}, we recall several results on modular curves, their Jacobian varieties, the Hecke operators and so on. In Section \ref{section : the structure of the kernel}, we develop some ideas used in previous work of Mazur, Ribet and the author in order to study the basic structure of the kernel of a rational Eisenstein prime. In Section \ref{section : proof}, we prove the theorems in the introduction.

\subsection*{Acknowledgements} I would like to thank Kenneth Ribet for very helpful discussions and correspondence over the years.

\section{Background} \label{section : background}
In this section, we review basic theory of modular curves and Hecke operators. For more details, see \cite[\ts 2]{Yoo6} and references therein.
\ms

For a positive integer $N>1$, consider the modular curve $X_0(N)$ over $\Q$.
It is the compactified coarse moduli scheme associated with the stack of the pairs $(E, C)$, where $E$ is an elliptic curve and $C$ is a cyclic subgroup of $E$ of order $N$.
We denote by $J_0(N)$ its Jacobian, which is an abelian variety over $\Q$.
\ms

Let $N=Mp^n$ for $n\geq 1$ with $(M, p)=1$. There are natural {\sf degeneracy maps} 
\begin{equation*}
\a_p, \b_p : X_0(N)=X_0(Mp^n) \to X_0(Mp^{n-1})=X_0(N/p)
\end{equation*}
which, in the level of stacks, consists of ``forgetting the level $p$ structure'' and ``dividing by the level $p$ structure,'' respectively. More precisely, they send $(E, C_M, C_p)$ to $(E, C_M, C_p[p^{n-1}])$ and $(E/{C_p[p]}, (C_M\oplus C_p[p])/{C_p[p]}, C_p/{C_p[p]})$, respectively, where $E$ is an elliptic curve and $C_M$ (resp. $C_p$) is a cyclic subgroup of $E$ of order $M$ (resp. $p^n$). They induce the following maps via two functorialities of Jacobians:
\begin{equation*}
\xymatrix{
J_0(N)   \ar@<.5ex>[r]^-{\a_{p,*}} \ar@<-.5ex>[r]_-{\b_{p, *}} & J_0(N/p) \ar@<.5ex>[r]^-{\a_p^*} \ar@<-.5ex>[r]_-{\b_p^*} & J_0(N)
}
\end{equation*}
and we denote by 
\begin{equation*}
\g_p :=\a_p^*+\b_p^* : J_0(N/p) \times J_0(N/p) \to J_0(N)
\end{equation*}
the map sending $(x, y)$ to $\a_p^*(x)+\b_p^*(y)$. The image of $\g_p$ is called the {\sf $p$-old subvariety of $J_0(N)$} and denoted by $J_0(N)_{p\hyp\old}$. The quotient of $J_0(N)_{p\hyp\old}$ is called the {\sf $p$-new quotient of $J_0(N)$}. 
The abelian subvariety of $J_0(N)$ spanned by the $p$-old subvarieties for all prime divisors $p$ of $N$ is called the {\sf old subvariety of $J_0(N)$} and denoted by $J_0(N)_{\old}$. The quotient of $J_0(N)$ by $J_0(N)_\old$ is called the {new quotient of $J_0(N)$} and denoted by $J_0(N)^\new$.
\ms

Now, for a prime number $p$ the {\sf $p$-th Hecke operator $T_p$} is defined by the composition:
\begin{equation*}
\xymatrix{
T_p = \b_{p, *} \circ \a_p^* : J_0(N) \ar[r]^-{\a_p^*} & J_0(Np) \ar[r]^-{\b_{p, *}} & J_0(N).
}
\end{equation*}
The {\sf $n$-th Hecke operator $T_n$} are defined inductively as follows:
\begin{itemize}
\item
$T_1=1$;
\item
$T_{mn}=T_m T_n$ if $(m, n)=1$;
\item
$T_{p^k}=T_p T_{p^{k-1}}-pT_{p^{k-2}}$ for $k\geq 2$.
\end{itemize}
The {\sf Hecke ring of level $N$}, denoted by $\T(N)$, is the commutative subring of $\End(J_0(N))$ generated (over $\Z$) by $T_n$ for all integers $n\geq 1$.
\ms
 
As before, let $N=Mp^n$ for $n\geq 1$ with $(M, p)=1$. For distinction, let $T_p$ (resp. $\tau_p$) denote the $p$-th Hecke operator acting on $J_0(Mp^n)$ (resp. $J_0(Mp^{n-1})$). Then by the relation on the $p$-old subvariety
\begin{equation}\label{equation : Up and Tp}
T_p=\mat {\tau_p} p{-1} 0 \text{ if } n = 1 \qa T_p=\mat {\tau_p} p 0 0 \text{ if } n \geq 2,
\end{equation}
the map $\g_p$ is Hecke-equivariant. (For a prime $q\neq p$, $T_q$ diagonally acts on the $p$-old subvariety.)
Moreover $J_0(N)_{p\hyp\old}$ and $J_0(N)^{p\hyp\new}$ are both stable under the action of the Hecke operators. The image of $\T(N)$ in $\End(J_0(N)_{p\hyp\old})$ (resp. $\End(J_0(N)^{p\hyp\new})$) is called the {\sf $p$-old (resp. $p$-new) quotient of $\T(N)$} and denoted by $\T(N)^{p\hyp\old}$ (resp. $\T(N)^{p\hyp\new}$). 
We say a maximal ideal $\fm$ of $\T(N)$ is {\sf $p$-new} if it is still maximal after the projection $\T(N) \surj \T(N)^{p\hyp\new}$. Analogously, $J_0(N)^\new$ is stable under the action of the Hecke operators, and the image of $\T(N)$ in $\End(J_0(N)^\new)$ is called the {\sf new quotient of $\T(N)$} and denoted by $\T(N)^\new$. A maximal ideal of $\T(N)$ is called {\sf new} if it is still maximal after the natural projection $\T(N) \surj \T(N)^\new$. Note that if a maximal ideal is new then it is $p$-new for all prime divisors $p$ of $N$.
\ms

Now, we recall some results on the cusps of $X_0(N)$. 
The cusps of $X_0(N)$ may be conveniently described using Shimura's notation \cite[\ts 1.6]{Shi71}.
Let $\G_0(N)\backslash \bP^1(\Q)$ be the cusps on $X_0(N)(\C)$. Then for $a/b \in \G_0(N)\backslash \bP^1(\Q)$ with $a, b \in \Z$ and relatively prime, we denote the associated cusp on $X_0(N)$ by $\bect a b$. 
The cusps of $X_0(N)$ then are represented by $\bect x d$, where $d$ is a divisor of $N$, $1\leq x \leq d$, and $(x, d)=1$ with $x$ taken  modulo $t:=(d, N/d)$. (Such a cusp $\bect x d$ is called a
{\sf cusp of level $d$}.) The cusp $\bect x d$ is defined over $\Q(\mu_t)$, so the action of $\Gal(\ov \Q/\Q)$ on $\bect x d$ factors through $\Gal(\Q(\mu_t)/\Q)$. The action of $\Gal(\Q(\mu_t)/\Q)$ permutes all the cusps of level $d$. (For details, see \cite[\ts 1]{Og73}, \cite[\ts 3.8]{DS05} or \cite[\ts 2]{Lg97}.)

The {\sf rational cuspidal subgroup of $J_0(N)$}, denoted by $\cC(N)$, is the subgroup of $J_0(N)$ generated by the degree $0$ divisors which are supported only on the cusps and are stable under the action of $\Gal(\ov \Q/\Q)$. (By Manin and Drinfeld \cite{Ma72, Dr73}, it is finite.)
 It is generated by the elements of the form
\begin{equation*}
\sum_{d \mid N} a_d (P_d), \text{ with } \sum_{d \mid N} a_d \cdot \p( (d, N/d))=0,
\end{equation*}
where the divisor $(P_d)$ is defined as the sum of all the cusps of level $d$ and $\p$ is the Euler's totient function. (From its definition, it is easy to see that $(P_d)$ is of degree $\p ((d, N/d))$ and stable under the action of $\Gal(\ov \Q/\Q)$.)
\ms

Let $N=\prod p^{e(p)}$ be the prime power factorization of $N$, and let
\begin{equation*}
N^\sqf:=\prod_{e(p)=1} p \qa N^\square:=\prod_{e(p)\geq 2} p.
\end{equation*}
For a divisor $M$ of $N^\sqf$ with $MN^\square \neq 1$, we define a cuspidal divisor $\cC_{M, N} \in \cC(N)$ by the image in $J_0(N)(\Q)$ of the divisor
\begin{equation}
C_{M, N}=\sum_{d \mid MN^\square} (-1)^{\o(d)} \p(N^\square/(d, N^\square))(P_d),
\end{equation}
where $\o(d)$ is the number of distinct prime divisors of $d$. Then, we have the following:

\begin{proposition}\label{proposition : cuspidal element}
The order of $\cC_{M, N}$ is the numerator of 
\begin{equation*}
\frac{h\prod_{p \mid M} (p-1) \prod_{p \mid N,~p \nmid M} (p^2-1) \prod_{p \mid N^\square} p^{e(p)-2}}{24},
\end{equation*}
where $h$ is either $1$ or $2$. For a prime $q$, the Hecke operator $T_q$ acts on $\cC_{M, N}$ by the multiplication by 
\begin{equation*}
\begin{cases}
~~1 & \text{ if $q$ divides $M$},\\
~~q & \text{ if $q$ divides $N^\sqf/M$},\\
~~0 & \text{ if $q^2$ divides $N$},\\
q+1 & \text{ if $q$ does not divide $N$}.
\end{cases}
\end{equation*}
\end{proposition}
If $N$ is squarefree (resp. otherwise), this proposition follows from \cite[Prop.  2.13]{Yoo1} and \cite[Th. 1.3]{Yoo3} (resp. from \cite[Th. 4.2 and 4.7]{Yoo6}).
\ms

Finally, we recall the following lemma on the maximality of $\fm_\ell(s, t, u)$.

\begin{lemma}\label{lemma : maximal ideal assumption}
Let $\ell\geq 5$ be a prime, and let $N=\prod_{i=1}^s p_i \prod_{j=1}^t q_j \prod_{k=1}^u r_k^{e(k)}$ with $e(k)\geq 2$ not divisible by $\ell$.
Then, $\fm:=\fm_\ell(s, t, u)$ is maximal if and only if $s+u\geq 1$ and $s_0+t+u_1 \geq 1$. 
\end{lemma}
\begin{proof}
Let $u=0$. Then, $N$ is squarefree and $\fm$ cannot be maximal if $s=0$ by \cite[Th. 1.2(1)]{Yoo2}. Therefore $s+u \geq 1$ if $\fm$ is maximal. 

Now, assume that $s+u \geq 1$.
Let $I:=(T_{p_i}-1, T_{q_j}-{q_j}, T_{r_k}, \cI_0(N) : \text{ for all } i, j \text{ and } k)$ be an Eisenstein ideal of $\T(N)$. Then by \cite[Th. 4.2 and 6.1]{Yoo6} we get
\begin{equation*}
(\T(N)/I )\otimes_\Z \Z_\ell \simeq (\zmod n)\otimes_\Z \Z_\ell
\end{equation*}
where $n$ is the numerator of 
\begin{equation*}
\frac{\prod_{i=1}^s (p_i-1) \prod_{j=1}^t (q_j^2-1) \prod_{k=1}^u r_k^{e(k)-2}(r_k^2-1)}{24}.
\end{equation*}
By our assumption, $q_j \equiv -1 \pmod \ell$ for all $j$ and hence $\fm= (\ell, I)$. Therefore, $\fm$ is maximal if and only if $n$ is divisible by $\ell$. Since $\ell \geq 5$, the latter statement is equivalent to saying that $s_0+t+u_1 \geq 1$, which implies the claim.
\end{proof}

\begin{remark}\label{remark : maximal ideal}
It is also true that for $M=\prod_{i=1}^s p_i$,
\begin{equation*}
\fm:=\fm_\ell(s, t, u) \text{ is maximal } \iff MN^\square \neq 1 \text{ and } \< \cC_{M, N} \>[\fm] \neq 0.
\end{equation*}
\end{remark}
\ms

\section{The structure of $J_0(N)[\fm]$} \label{section : the structure of the kernel}
Let $\fm$ be a rational Eisenstein prime of $\T(N)$.
The structure of $J_0(N)[\fm]$ is carefully studied by Mazur when $N$ is a prime \cite{M77}; by Ribet and the author when $N$ is squarefree and $\ell$ does not divide $6N$ \cite{RY, Yoo1}. Most of the ideas can be generalized to arbitrary composite level without difficulties if we assume that $\ell$ does not divide $2N$. For the convenience of readers, we try to provide enough details.
\ms

From now on, we fix a prime $\ell \geq 3$ and assume that $N$ is a positive integer prime to $\ell$.
For ease of notation, we set $\T:=\T(N)$ and $J:=J_0(N)$. 
Let $\fm$ denote a rational Eisenstein prime of $\T$ containing $\ell$. (Note that since all Hecke operators are congruent to an integer modulo $\fm$ (cf. \cite[Lemma 2.2]{Yoo6}), $k_\fm:=\T/\fm \simeq \F_\ell$.) Let $S=\Spec~\Z$ and $S':=\Spec ~\Z[1/N]$.
\ms

Let $\cJ_{\Z}$ denote the N\'eron model of $J$ over $S$. The description of the special fiber can be given by the theory of Raynaud \cite{Ra70} using the Deligne-Rapoport \cite{DR73} (resp. Katz-Mazur \cite{KM85}) model of $X_0(N)_{\F_p}$ when $p$ exactly divides $N$ (resp. $p^2$ divides $N$). If $p$ does not divide $N$, $\cJ_{\Z_p}$ is an abelian variety by Igusa \cite{Ig59}. Note that $\cJ_{\Z_p}$ is a semi-abelian variety if $p$ exactly divides $N$ 
and it is not if $p^2$ divides $N$.
Let $J[\ell]_{S}$ denote the scheme-theoretic kernel of multiplication by $\ell$ in $\cJ_{\Z}$. Then, $J[\ell]_{S}$ is a quasi-finite flat group scheme, whose restriction to $S'$ is finite and flat, because $\ell^2$ does not divide $N$ (cf. \cite[\ts 7.3, Lemma 2]{BLR90}). 

Now, we define $J[\fm]_{S}$ to be the Zariski-closure of $J[\fm]$ in $\cJ_{\Z}$.
It is the subgroup scheme extension of $J[\fm]$ in $J[\ell]_{S}$ in the sense of \cite[Ch. I, \ts 1(c)]{M77}.
(Note that $J[\fm] \subset J[\ell]_{S}(\ov{\Q})$ because $\ell \in \fm$.)
Therefore $J[\fm]_{S}$ is a quasi-finite flat group scheme, whose restriction to $S'$ is finite and flat. It is
a closed subgroup scheme of $\cJ_{\Z}$ by construction, and killed by $\fm$. The quotient $\T/\fm$ acts naturally on $J[\fm]_{S}$. We often abuse notation and write $J[\fm]$ for this quasi-finite flat group scheme $J[\fm]_{S}$ if there is no confusion. For related discussions, see \cite[Ch. I, \ts 1]{M77}, \cite[\ts 1]{M78}, \cite[\ts 3.1]{BK14} or \cite{Con03}.

\subsection{Mazur's argument} \label{section : Mazur's argument}
The main result in this subsection is the following:
\begin{proposition}\label{proposition : mazur argument}
There is an exact sequence:
\begin{equation*}
\xymatrix{
0 \ar[r] & \zmod \ell \ar[r] & J[\fm] \ar[r] & \mu_\ell^{\oplus n} \ar[r] & 0
}
\end{equation*}
with $n\geq 1$.
\end{proposition}
The following is an easy consequence of the proposition:
\begin{corollary} \label{corollary : unipotent nature}
Let $p$ be a prime different from $\ell$ and let $I_p$ be an inertia subgroup of $G_\Q$ for $p$. Then, for any $\s_p \in I_p$, $(\s_p-1)^2$ annihilates $J[\fm]$.
\end{corollary}
\begin{proof}
By definition, $(\s_p-1)$ annihilates $J[\fm]^{I_p}$, the fixed part of $J[\fm]$ by $I_p$. Then, by Proposition \ref{proposition : mazur argument} the quotient $Q=J[\fm]/{J[\fm]^{I_p}}$ is isomorphic to $\mu_\ell^{\oplus a}$ for some $0\leq a\leq 1$, which is unramified at $p$. Therefore $(\s_p-1)$ also annihilates $Q$ which implies the claim.
\end{proof}

The proof of Proposition \ref{proposition : mazur argument} will directly follow from the discussion below.
\begin{lemma} \label{lemma : dim at least 2}
The dimension of $J[\fm]$ over $k_\fm$ is at least $2$.
\end{lemma}
\begin{proof}
We apply the arguments in \cite[Ch. II, \ts\ts 6, 7]{M77}.

For an ideal $\fa$ of $\T$, let $J[\fa]:=\{ x \in J(\ov{\Q}) : Tx=0 \text{ for all } T \in \fa \}$.
Let 
\begin{equation*}
\T_\fm:=\plim k \T/{\fm^k} \qa J_\fm :=\ilim k J[\fm^k].
\end{equation*}
Then, $\T_\fm$ is a direct factor of the semi-local ring $\T_\ell:=\T \otimes \Z_\ell$
and $J_\fm \simeq J_\ell \otimes_{\T_\ell} \T_\fm$, where $J_\ell:=\ilims k J[\ell^k]$, the $\ell$-divisible group associated to $J$.
Let 
\begin{equation*}
\Ta_{\fm} (J):=\Hom (\Q_\ell/\Z_\ell, ~J_\fm)=\Hom (\Q_\ell/\Z_\ell, ~\ilim k J[\fm^k])
\end{equation*}
be the $\fm$-adic Tate module of $J$, which is isomorphic to $\Ta_\ell(J) \otimes_{\T_\ell} \T_\fm$.
Then by Lemma 7.7 of \textit{loc. cit.}
 $\Ta_{\fm} (J) \otimes_\Z \Q$ is of dimension 2 over $\T_\fm \otimes_\Z \Q$. Note that the argument used in the proof is not dependent on the fact that the level $N$ is a prime (cf. \cite[Prop.  3.1]{Wi80}, \cite[p. 481]{Wi95} or \cite[Lemma 1.39]{DDT}).
 
Since $\fm$ is maximal and $\T$ acts faithfully on $J$, $J[\fm] \neq 0$.
Suppose that $\dim_{k_\fm} J[\fm] =1$. Then by Nakayama's lemma $\Ta_{\fm}(J)$ is free of rank $1$ over $\T_\fm$. Thus, $\Ta_{\fm} (J)  \otimes_\Z \Q$ is of dimension $1$ over $\T_\fm \otimes_\Z \Q$, which is a contradiction. Therefore $\dim_{k_\fm} J[\fm] \geq 2$ as desired.
\end{proof}
\begin{remark}
The proof above is valid without any hypothesis on $\ell$ and $\fm$. Thus, Lemma \ref{lemma : dim at least 2} holds for any maximal ideal $\fm$ of $\T$.
\end{remark}

The following is a generalization of \cite[Ch. II, Prop. 14.1]{M77}:
\begin{proposition}\label{proposition : jordan-holder factors}
All Jordan-H\"older constituents of $J[\fm]$ is either $\zmod \ell$ or $\mu_\ell$. 
\end{proposition}
\begin{proof}
We apply the argument on pages 114--115 of \cite{M77}. 

Let $W=J[\fm]^\text{ss}$ be the semisimplification of $J[\fm]$, and let $W^\vee$ be the Cartier dual of $W$. Let $d:=\dim J[\fm]$ as a vector space over $\T/\fm \simeq \F_\ell$, and let $G$ denote a finite quotient of $G_\Q$ through which the action on $W$ factors. 
Let $\Frob_p$ denote a Frobenius element for $p$ in $G_\Q$. 
Then for any prime number $p$ not dividing $\ell N$, by the Eichler-Shimura relation the action of $\Frob_p$ on $W$ satisfies
\begin{equation*}
\Frob_p^2-(T_p ~ \mod \fm) \Frob_p + p =0.
\end{equation*}
Since $T_p-p-1 \in \fm$, the only eigenvalues possible for the action of $\Frob_p$ on $W$ are either $1$ or $p$. Since Cartier duality interchanges these eigenvalues, the characteristic polynomial of $\Frob_p$ on $W\oplus W^\vee$ is 
\begin{equation*}
(X-1)^d(X-p)^d.
\end{equation*}
Note that the characteristic polynomial of $\Frob_p$ on $(\zmod \ell \oplus \mu_\ell)^{\oplus d}$ is $(X-1)^d (X-p)^d$ as well. 
By the Chebotarev density theorem any element in $G$ is the image of some $\Frob_p$ (for a prime $p$ not dividing $\ell N$) and hence any element $g\in G$ has the same characteristic polynomial for the representation $W\oplus W^\vee$ as for $(\zmod \ell\oplus \mu_\ell)^{\oplus d}$. Since both representations are semisimple, they are in fact isomorphic by the Brauer-Nesbitt theorem. Therefore 
\begin{equation*}
J[\fm]^\text{ss}=W \simeq (\zmod \ell)^{\oplus a} \oplus \mu_\ell^{\oplus b}
\end{equation*}
for some integers $a, b \geq 0$ with $a+b=d$ as desired.
\end{proof}

\begin{remark}
The proof above is valid without any hypothesis on $\ell$. Thus, Proposition \ref{proposition : jordan-holder factors} holds for any rational Eisenstein prime $\fm$ of $\T$.
\end{remark}

\begin{lemma} \label{lemma : cuspidal in J[m]}
$\cC(N)[\fm] \simeq \zmod \ell \subset J[\fm]$.
\end{lemma}
\begin{proof}
The fact that $\cC[\fm]\neq 0$ follows from \cite[Th. 1.3(4)]{Yoo6}. More specifically, if we write 
\begin{equation*}
N=\prod_{i=1}^s p_i \prod_{j=1}^t q_j \prod_{k=1}^u r_k^{e(k)} \text{ with } e(k)\geq 2
\end{equation*}
and if $\fm:=\fm_\ell(s, t, u)$, then by Proposition \ref{proposition : cuspidal element} $\<\cC_{M, N}\>[\fm] \simeq \zmod \ell$, where $M=\prod_{i=1}^s p_i$. Since $\cC(N) \subset J(\Q)$, $\cC(N)[\fm] \simeq (\zmod \ell)^{\oplus a}$ for some $a\geq 1$.

Now, we will deduce $a=1$ from the arguments which parallel those given on \cite[pp. 118--119]{M77}.
Since $N$ and $\ell$ are relatively prime, we can consider $X_0(N)$ as a curve over $\F_\ell$, which we denote by $X_0(N)_{\F_\ell}$  \cite{Ig59}.  
Consider reductions of $J[\ell]_{S'}$ and $J[\fm]_{S'}$ at $\ell$, and denote them by $J[\ell]_{\F_\ell}$ and $J[\fm]_{\F_\ell}$, respectively. (See the beginning of this section.) Note that 
$J[\ell]_{\F_\ell}$ and $J[\fm]_{\F_\ell}$ are both finite and flat.
Let $J[\ell]_{\F_\ell}^{\text{\'et}}$ and $J[\fm]_{\F_\ell}^{\text{\'et}}$ denote the \'etale parts of $J[\ell]_{\F_\ell}$ and $J[\fm]_{\F_\ell}$, respectively.
Recall the canonical isomorphism
\begin{equation*}
\d : J[\ell](\ov{\F}_\ell) \to H^0(X_0(N)_{\ov{\F}_\ell}, \O^1)^{\mathscr{C}}
\end{equation*}
of \cite[\ts 11, Prop. 10]{Se58}, where the superscript $\mathscr{C}$ means fixed elements under the Cartier operator. 
This map $\d$ is defined as follows: an element $x \in J[\ell](\ov{\F}_\ell)$ is represented by a divisor $D$ on $X_0(N)_{{\ov{\F}_\ell}}$ such that $\ell D=(f)$. One takes $\d(x)=df/f$. Then, the isomorphism above induces an injection:
\begin{equation*}
\d : J[\ell](\ov{\F}_\ell) \otimes_{\F_\ell} \ov{\F}_\ell \inj H^0(X_0(N)_{\ov{\F}_\ell}, \O^1)
\end{equation*}
which commutes with the action of $\T/{\ell \T}$ (cf. \cite[Ch. II, Prop. 14.7]{M77}). Also, the injective $q$-expansion map $H^0(X_0(N)_{\ov{\F}_\ell}, \O^1) \to \ov{\F}_\ell[[q]]$ and the arguments of  \cite[Ch. II, \ts 9]{M77} show that $H^0(X_0(N)_{\ov{\F}_\ell}, \O^1)[\fm]$ is at most of dimension $1$. Therefore $(J[\ell]_{\F_\ell}^{\text{\'et}})[\fm]$ is of dimension $\leq 1$. Note that $\fm$ is ordinary because $T_\ell\equiv \ell+1 \equiv 1 \pmod \fm$. Since $\ell \geq 3$, we get
\begin{equation*}
J[\fm]_{\F_\ell}^{\text{\'et}} = (J[\ell]_{\F_\ell}^{\text{\'et}})[\fm]
\end{equation*}
(cf. \cite[Ch. II, Cor. 14.8]{M77}) and this shows that $J[\fm]$ can have at most one $\zmod \ell$ as its Jordan-H\"older constituents. Thus, $a=1$ as claimed.
\end{proof}

\begin{proposition} \label{proposition : criterion of constancy}
Let $M:=J[\fm]/{(\zmod \ell)}$. Then $M \simeq \mu_\ell^{\oplus n}$ for some $n\geq 1$.
\end{proposition}
\begin{proof}
The claim $n\geq 1$ directly follows from Lemma \ref{lemma : dim at least 2} if $M \simeq \mu_\ell^{\oplus n}$. 

In the proof of Lemma \ref{lemma : cuspidal in J[m]}, we proved that $M^\text{ss} \simeq \mu_\ell^{\oplus n}$.
To do that $M \simeq \mu_\ell^{\oplus n}$, we follow the argument on page 125 of \cite{M77} (or \cite[Prop.  4.2]{Wi80}, \cite[Rmk. 4.2.4]{BK14}).

Let $p$ be a prime not dividing $N\ell$. Then, $T_p$ acts on $M$ by $p+1$ because $T_p-p-1 \in \fm$.
By the Eichler-Shimura relation, for each prime $p$ not dividing $\ell N$ a Frobenius element $\Frob_p$ acts on $M$ by
\begin{equation*}
\Frob_p^2-(p+1) \Frob_p+p=(\Frob_p-1)(\Frob_p-p)=0.
\end{equation*}
Note that $M$ is of multiplicative type, so the only eigenvalue of $\Frob_p$ acting on $M$ is $p$. Therefore, if $p\not\equiv 1\pmod \ell$, then $(\Frob_p-1)$ is an isomorphism on $M$ and $(\Frob_p-p)$ annihilates $M$ by the above formula. 
Let $M^\vee$ be the Cartier dual of $M$, then $M^\vee$ is \'etale and $\Frob_p$ acts trivially on $M^\vee$ for any prime $p \nmid N\ell$ with $p\not\equiv 1 \pmod \ell$. By the Chebotarev density theorem, $M^\vee \simeq (\zmod \ell)^{\oplus n}$ and hence $M \simeq \mu_\ell^{\oplus n}$.
\end{proof}
\ms

\subsection{The Shimura subgroup}
Let $X_1(N)$ be the compactified coarse moduli scheme associated with the stack of the pairs $(E, P)$, where $E$ is an elliptic curve and $P$ is a point of exact order $N$.
The natural homomorphism from $X_1(N)$ to $X_0(N)$ of modular curves, which sends a point $(E, P) \in Y_1(N)$ to $(E, \< P \>)$, induces the morphism $J_0(N) \to J_1(N)$ by Picard functoriality. The kernel of this morphism is called the {\sf Shimura subgroup of $J_0(N)$}, denoted by $\S(N)$.
In their paper \cite{LO91}, Ling and Oesterl\'e gave a complete description of the Shimura subgroup: the structure as an abstract abelian group, the exponent, the order, the action of the Galois group $\Gal(\ov{\Q}/\Q)$, the action of the Hecke operators $T_p$ for all primes $p$ and so on. 
Roughly speaking, $\S(N)$ is a quotient of $(\zmod N)^\times$ by a small subgroup of order $n$, where $n$ is a product of powers of $2$ and $3$. 
The smallest common field of definition of the points of $\S(N)$ is the cyclotomic field $\Q(\mu_e)$, where $e$ is the exponent of $\S(N)$. The Galois group $\Gal(\Q(\mu_e)/\Q)$ acts on $\S(N)$ via the cyclotomic character $\Gal(\Q(\mu_e)/\Q) \to (\zmod e)^\times$.
For a prime $p$ not dividing $N$ (resp. dividing $N$), $T_p$ acts by $p+1$ (resp. $p$) on $\S(N)$.
Combining all the results, we get the following:
\begin{proposition}[Ling-Oesterl\'e] \label{proposition : Ling-Oesterle}
Let $N=\prod_{p} p^{e(p)}$ be the prime power decomposition of $N$. 
\begin{enumerate}
\item
If $T_p \not \equiv p \pmod \fm$ for some $p$, then $\S(N)[\fm]=0$.
\item
Suppose that $T_p \equiv p \pmod \fm$ for all $p$.
\begin{enumerate}
\item
If $\ell \geq 5$, then 
$\S(N)[\fm] \simeq \mu_\ell^{\oplus a}$, where $a$ is the number of $p$'s which are congruent to $1$ modulo $\ell$.

\item
Suppose further that $\ell=3$.
If $p \equiv 1 \pmod 3$ for all $p$, then 
$\S(N)[\fm] \simeq \mu_3^{\oplus b}$, where $b$ is the number of $p$'s which are congruent to $1$ modulo $9$. If $p \equiv -1 \pmod 3$ for some $p$, then $\S(N)[\fm] \simeq \mu_3^{\oplus c}$, where $c$ is the number of $p$'s which are congruent to $1$ modulo $3$.
\end{enumerate}
\end{enumerate}
\end{proposition}
In Mazur's paper \cite[Th. 2]{M77}, he showed that $\S(N)$ is the maximal $\mu$-type subgroup of $J$ is the Shimura subgroup when $N$ is a prime. This result is generalized by Vatsal \cite{Va05} to composite level, and the following is a special case of his result:
\begin{proposition}[Vatsal] \label{proposition : vatsal}
If $\mu_\ell \subset J$ then $\mu_\ell \subset \S(N)$.
\end{proposition}
Note that since we assume that $\ell$ does not divide $2N$, $\ell$ is odd and $J$ has a good reduction at $\ell$. Thus, this is a direct consequence of \cite[Th. 1.1]{Va05}.
\begin{definition} \label{definition : cQm}
We denote by $\cQ[\fm]$ the quotient of $J[\fm]$ by $\S(N)[\fm]$. Also, we denote by $\cQ[\fm]_{S'}$ its associated finite flat group scheme over $S'$, which is defined as the quotient of $J[\fm]_{S'}$ by $\S(N)[\fm]_{S'}$.
\end{definition}

\begin{corollary}\label{corollary : structure of cQ[m]}
The quotient $\cQ[\fm]$ is an extension of $\mu_\ell^{\oplus k}$ by $\zmod \ell$ for some $k \geq 0$ such that $\mu_\ell \not\subset \cQ[\fm]$.
\end{corollary}
\noindent (Here, $k$ could be zero, for instance, if $N$ is a prime.)
\begin{proof}
Since $\S(N)[\fm] \simeq \mu_\ell^{\oplus a}$ by Proposition \ref{proposition : Ling-Oesterle}, the first claim follows from Proposition \ref{proposition : mazur argument}.

Now suppose that $\mu_\ell \subset \cQ[\fm]$. Let $\pi : J[\fm] \to \cQ[\fm]$ be the quotient map.
Then, $\pi^{-1}(\mu_\ell) \subset J[\fm]$ is an extension of $\mu_\ell$ by $\S(N)[\fm] \simeq \mu_\ell^{\oplus a}$. Since $\pi^{-1}(\mu_\ell)$ is annihilated by $\fm$ as well, by the same argument as in the proof of Proposition \ref{proposition : criterion of constancy}
we get
\begin{equation*}
\pi^{-1}(\mu_\ell) \simeq \mu_\ell^{\oplus{(a+1)}} \simeq \S(N)[\fm] \oplus \mu_\ell \subset J[\fm],
\end{equation*}
which contradicts Proposition \ref{proposition : vatsal}. Therefore $\mu_\ell \not\subset \cQ[\fm]$.
\end{proof}
\begin{remark} \label{remark : prosaic nugget}
Thanks to Corollary \ref{corollary : unipotent nature}, 
even though $J$ does not have semistable reduction at $p$ if $p^2$ divides $N$, 
both $J[\fm]_{S'}$ and $\cQ[\fm]_{S'}$ belong to the category $\ud{D}$ of Schoof \cite[p. 2468]{BK14}\footnote{The category $\ud{D}$ is first introduced by Schoof \cite[\ts 2]{Sch05} when $N$ is a prime.}. Moreover, $\cQ[\fm]_{S'}$ with a filtration $0 \subset (\zmod \ell)_{S'} \subset \cQ[\fm]_{S'}$ is an example of a prosaic nugget in \textit{loc. cit.} If $N$ is not squarefree and $\fm:=\fm_\ell(s, t, u)$ then $\S(N)[\fm]=0$ and hence $J[\fm]_{S'}=\cQ[\fm]_{S'}$ is also a prosaic nugget.
\end{remark}
\ms

\subsection{Ramification of $J[\fm]$}
In this section, we provide a criterion for ramification of $J[\fm]$:
\begin{proposition} \label{proposition : ramification semistable}
Let $p\neq \ell$ be a prime whose square does not divide $N$. If $T_p \not\equiv p \pmod \fm$, then $J[\fm]$ is unramified at $p$, i.e., $I_p$ acts trivially on $J[\fm]$, where $I_p$ is an inertia subgroup of $\Gal(\ov{\Q}/\Q)$ for $p$.
\end{proposition}
In the paper \cite{RY}, Ribet and the author proved the preceding proposition when $N$ is squarefree. The idea is to use Grothendieck's semistable reduction theorem \cite{Gro72} and the structure of $J[\fm]$ in Section \ref{section : Mazur's argument}. The argument applies \textit{mutatis mutandis} as long as 
$J$ has semistable reduction at $p$, which is equivalent to saying that $p^2$ does not divide $N$.

For the convenience of readers, we give a complete proof.
\begin{proof}
If $p$ does not divide $N$, then $J$ has a good reduction at $p$ and $I_p$ acts trivially on $J[\ell]$. Since $J[\fm] \subset J[\ell]$, the result follows. 

Now, we assume that $p$ exactly divides $N$, i.e., $p$ divides $N$ but $p^2$ does not divide $N$.
Let $Q=J[\fm]/{J[\fm]^{I_p}}$ and it suffices to show that $Q=0$ if $T_p \not\equiv p \pmod \fm$.

First of all, we claim that $Q$ is unramified at $p$ and $T_p$ acts on $Q$ by $\Frob_p$. 
Let $\cJ_{\F_p}$ denote the special fiber of the N\'eron model of $J$ over $\F_p$. 
Then, by the Deligne-Rapoport model of $X_0(N)_{\F_p}$ \cite{DR73} and the theory of Picard functors by Raynaud \cite{Ra70}, $\cJ_{\F_p}$ is an extension of the component group by a semiabelian variety $\cJ^0_{\F_p}$, called the identity component. Moreover $\cJ^0_{\F_p}$ is an extension of an abelian variety $A_{\F_p}$, which is isomorphic to $J_0(N/p)_{\F_p}\times J_0(N/p)_{\F_p}$, by a torus $\cT_{\F_p}$. If the rank of $\cT_{\F_p}$ is $a$ and the dimension of $J_0(N/p)$ is $b$, then the dimension of $J$ is $a+2b$ and $J[\ell] \simeq (\zmod \ell)^{2(a+2b)}$. The $\ell$-torsion subgroup $J[\ell]$ has a natural filtration \cite[11.6.5]{Gro72}:
\begin{equation*}
0 \subset J[\ell]^t \subset J[\ell]^0 \subset J[\ell],
\end{equation*}
where $J[\ell]^0$ is its connected component over $\Z_p$ and $J[\ell]^t$ is the ``toric part'' of $J[\ell]^0$.
Note that $J[\ell]^0 \simeq \cJ^0_{\F_p}[\ell]$ is of rank $a+4b$ and $J[\ell]^t \simeq \cT_{\F_p}[\ell]$ is of rank $a$. Moreover by Serre and Tate \cite[Lemma 2]{ST68}, $J[\ell]^{I_p} \simeq \cJ_{\F_p}[\ell]$, which implies that $J[\ell]^0 \subset J[\ell]^{I_p}$.\footnote{In \cite[\ts 1.(b)]{M78}, Mazur writes $FJ[\ell]$ for the finite part of $J[\ell]$. Using \cite[Ch I, Prop. 1.3]{M77} or its generalization \cite[Cor. 1.5]{Con03}, we get $J[\ell]^{I_p} = FJ[\ell]$.}
The quotient $J[\ell]/{J[\ell]^0}$ (of rank $a$) is unramified and canonically isomorphic to $X/{\ell X}$, where $X$ is the character group of the torus of the special fiber of the N\'eron model of the dual abelian variety $J^\vee$ over $\F_p$ (cf. \cite[Cor. of Prop. 11]{MR91}). By autoduality of $J$, $X$ is isomorphic to the character group of $\cT_{\F_p}$. By \cite[Prop.  3.8]{R90}, the action of $T_p$ on $X$ is equal to that of $\Frob_p$ on $X$. Therefore the actions of $T_p$ and $\Frob_p$ on $J[\ell]/{J[\ell]^0}$ coincide, and the same is true on $J[\ell]/{J[\ell]^{I_p}}$.
Since $J[\fm] \subset J[\ell]$ and $J[\fm]^{I_p}=J[\fm] \cap J[\ell]^{I_p}$, the actions of $T_p$ and $\Frob_p$ on $Q=J[\fm]/{J[\fm]^{I_p}}$ coincide, which proves the claim.

Secondly, we show that $\Frob_p$ acts on $Q$ by $p$.
Since $J[\fm]$ is an extension of $\mu_\ell^{\oplus n}$ by $\zmod \ell$, $J[\fm]^{I_p}$ is at least of dimension $n$ and $Q=J[\fm]/{J[\fm]^{I_p}}$ is isomorphic to either $\mu_\ell$ or $0$. Therefore the claim follows.

Finally, since both $(T_p-\Frob_p)$ and $(\Frob_p-p)$ annihilate $Q$, so does $T_p-p$.  If $T_p-p \not\in \fm$, then it is a unit in $\T/{\fm}$ and hence it acts as an automorphism on $Q=Q[\fm]$. Therefore $Q=0$.
\end{proof}

\begin{corollary} \label{corollary : unramified of Q[m]}
Let $p\neq \ell$ be a prime whose square does not divide $N$. Then,
the quotient $\cQ[\fm]$ is unramified at $p$ if $T_p \not \equiv p \pmod \fm$.
\end{corollary}
\begin{proof}
Since $\cQ[\fm]$ is the quotient of $J[\fm]$, it is clear.
\end{proof}
\ms

\subsection{The group of extensions}\label{section : the group of extensions}
In the rest of this section, we further assume that $\ell \geq 5$. 
\ms

For a prime $p\neq \ell$, Schoof shows that
\begin{equation*}
\dim_{\F_\ell} \Ext_{\Z[1/p]}(\mu_\ell, \zmod \ell) = \begin{cases}
1 & \text{ if } p \equiv \pm 1 \pmod \ell \\
0 & \text{ otherwise}
\end{cases}
\end{equation*}
where $\Ext_{\Z[1/p]}(\mu_\ell, \zmod \ell)$ denotes the group of extensions of $\mu_\ell$ by $\zmod \ell$ over $\Z[1/p]$ (cf. \cite[Cor. 4.2]{Sch05}). Brumer and Kramer generalize this result to any composite number $N$ prime to $\ell$. 
Let
\begin{equation*}
\cS(N) := \{ p \mid N :  p \equiv \pm 1 \pmod \ell \}
\end{equation*}
and let $\vp_\ell(N):=\# \cS(N)$.
The following is Proposition 4.2.1 of \cite{BK14}:
\begin{proposition}[Brumer-Kramer] \label{proposition : brumer-kramer}
$\dim_{\F_\ell} \Ext_{\Z[1/N]} (\mu_\ell, \zmod \ell) = \vp_\ell(N)$.
\end{proposition}
\ms

For an extension $E \in \Ext_{\Z[1/N]}(\mu_\ell, \zmod \ell)$, we define $\Q(E)$ by the field defined by the points of $E(\ov{\Q})$. If $N=p \equiv \pm 1 \pmod \ell$ is a prime, then 
 there is a basis element $E \in \Ext_{\Z[1/p]}(\mu_\ell, \zmod \ell)$ by Schoof. Moreover $\Q(E)$ is a unique extension field of $\Q(\mu_\ell)$ of degree $\ell$ (up to isomorphism) such that
\begin{itemize}
\item
$\Q(E)/{\Q(\mu_\ell)}$ is unramified outside the primes above $p$ and split over $\l$, where $\l$ is the unique prime of $\Q(\mu_\ell)$ above $\ell$.
\item
$\Gal(\Q(\mu_\ell)/\Q)$ acts on $\Gal(\Q(E)/{\Q(\mu_\ell)})$ by the inverse of the mod $\ell$ cyclotomic character.
\end{itemize}
(cf. \cite[Lemma 3.3.10]{BK14}.)

\begin{notation} \label{notation : extension at p}
From now on, for a prime $p\equiv \pm 1 \pmod \ell$, we choose a basis of $\Ext_{\Z[1/p]}(\mu_\ell, \zmod \ell)$ and denote it by $\cE_p$. Note that $\cE_p$ can be regarded as an extension over $\Z[1/N]$ if $N$ is a multiple of $p$. 
In such a case, for an extension $E \in \Ext_{\Z[1/N]}(\mu_\ell, \zmod \ell)$, we denote by
\begin{equation*}
 E \sim \cE_p
\end{equation*} 
if there is a $k \in \F_\ell^\times$ such that $E \simeq k\cE_p$ as elements in $\Ext_{\Z[1/N]}(\mu_\ell, \zmod \ell)$.  
Note that if $E \sim \cE_p$, then $\Q(E) \simeq \Q(\cE_p)$, so we set $\cF_p := \Q(\cE_p)$, which is unique up to isomorphism.
\end{notation}
\ms

Now, we interpret Proposition \ref{proposition : brumer-kramer} using characters on a certain Galois group. We use the same notation as in the proof of \cite[Prop. 4.2.1]{BK14}. Let $L$ be the maximal elementary abelian $\ell$-extension of $F:=\Q(\mu_\ell)$ such that $L/F$ is unramified outside $N$ and split over $\l_F=\l$.
Let $G:=\Gal(L/\Q)$, $G_0:=\Gal(L/F)$ and $\D:=\Gal(F/\Q)$. For each prime $p \in \cS(N)$, let $G_p := \Gal(\cF_p/F)$.
Then, we have
\begin{equation}\label{equation : Phi}
\xyh{1}
\xymatrix{
\Phi : \Ext_{\Z[1/N]}(\mu_\ell, \zmod \ell) \ar[r]^-{\simeq} & \Ext^1_{\F_\ell[G]} (\mu_\ell, \F_\ell)  \ar[r]^-{\simeq} & H^1 (G, \F_\ell(-1))  \ar[r]^-{\simeq} & \Hom_{\F_\ell} (G_0, \F_\ell(-1))^{\D}.
}
\end{equation}
For each $E \in \Ext_{\Z[1/N]}(\mu_\ell, \zmod \ell)$, by a suitable choice of basis (as an $\F_\ell$-vector space), we can obtain a Galois representation of the form:
\begin{equation*}
\rho_E : \Gal(\Q(E)/\Q) \ra \left\{\bmat {\mathbf{1}} {c_E} 0 {\chi_\ell} : c_E \in H^1(\Gal(\Q(E)/\Q), \F_\ell(-1)) \right\}.
\end{equation*}
Then by the restriction of $c_E$ to $\Gal(\Q(E)/F)$ we get a character\footnote{Note that $\mat a b 0 d \mat {\mathbf{1}} {x} 0 {\mathbf{1}} \mat a b 0 d ^{-1}= \mat {\mathbf{1}} {ad^{-1}\cdot x} 0 {\mathbf{1}}$.} $\iota_E : \Gal(\Q(E)/F) \to \F_\ell(-1)$ 
(cf. \cite[Rmk 3.5.4]{BK14}).
The composition of the canonical projection $G_0 \surj \Gal(\Q(E)/F)$ and $\iota_E$ gives rise to the desired homomorphism $\Phi(E)$.

By consideration of ramification, 
the set $\{ \Phi(\cE_p) : p \in \cS(N) \}$ forms a basis of $\Hom_{\F_\ell} (G_0, \F_\ell(-1))^{\D}$.
In other words, for any $E \in \Ext_{\Z[1/N]}(\mu_\ell, \zmod \ell)$ we can find $a_p \in \F_\ell$ such that
\begin{equation*}
\Phi(E)=\sum_{p \in \cS(N)} a_p \cdot \Phi(\cE_p)
\end{equation*}
and $a_p =0$ for all $p\in \cS(N)$ if and only if $E \simeq \zmod \ell \oplus \mu_\ell$.
By the same discussion as above, we get the following:
\begin{corollary}\label{corollary : extension basis}
Suppose that $\cS(N)=\{ p_1, \dots, p_n \}$. For each $i$, let $E_i \in \Ext_{\Z[1/N]}(\mu_\ell, \zmod \ell)$ such that $E_i \sim E_{p_i}$. Then, for any $E \in \Ext_{\Z[1/N]}(\mu_\ell, \zmod \ell)$ there is an element
$(a_1, \dots, a_n) \in \F_\ell^n$ such that
\begin{equation*}
\Phi(E)=\sum_{p \in \cS(N)} a_p \cdot \Phi(E_p).
\end{equation*}
Moreover, we have
\begin{equation*}
(a_1, \dots, a_n)=(0, \dots, 0) \iff E \simeq \zmod \ell \oplus \mu_\ell.
\end{equation*}
\end{corollary}
\ms

\begin{remark}
Since $L/F$ is an elementary abelian $\ell$-extension, by Kummer theory 
\begin{equation*}
 G_0=\Gal(L/F) \simeq G_1 \times \cdots \times G_k
\end{equation*}
for some $k$ with $G_i \simeq \zmod \ell$ (cf. \cite[Cor. of Lemma 3]{Bir67}). Indeed $\Phi(E) \in \Hom_{\F_\ell} (G_0, \F_\ell(-1))^{\D}$ can be decomposed as follows:
\begin{equation*}
\Phi(E) : G_0 \surj \prod_{p \in \cS(N)} G_p \to \F_\ell(-1).
\end{equation*}
\end{remark}
\ms

\subsection{The dimension of $J[\fm]$}
As before, we assume that $\ell \geq 5$. 
Let
\begin{equation*}
\cS_\fm := \{ p \in \cS(N) : p^2 \mid N \} \cup \{ p \in \cS(N) : p^2 \nmid N \text { and } T_p-p \in \fm \}.
\end{equation*}
Note that by Lemma \ref{lemma : maximal ideal assumption}, $\cS_\fm$ is non-empty because $\fm$ is maximal.
The following is a generalization of \cite[Th. 4.4]{Yoo1}:
\begin{theorem}\label{theorem : upper bound of dimension}
As a vector space over $\T/{\fm} \simeq \F_\ell$, we have 
\begin{equation*}
\dim_{\F_\ell} \cQ[\fm] \leq 1+\# \cS_\fm.
\end{equation*}
\end{theorem}
\begin{proof}
Let $S$ be the set of all prime divisors $p$ of $N$, and let $S_0$ be the subset of $S$ such that $p^2$ does not divide $N$ and $T_p-p \not\in \fm$. Let $N_0$ be the product of all elements in $S$ but not in $S_0$, i.e., 
\begin{equation*}
N_0:=\prod_{p \in S \sm S_0} p.
\end{equation*}
Then by Proposition \ref{proposition : ramification semistable} the Galois module $J[\fm]$ is unramified outside the primes dividing $\ell N_0$ and so is $\cQ[\fm]$. Therefore the corresponding group schemes $J[\fm]_{S'}$ and $\cQ[\fm]_{S'}$ prolong to a finite and flat scheme over $S''$, where $S''=\Spec ~\Z[1/{N_0}]$ (cf. \cite[Ch. I, Prop. 1.3]{M77} or \cite[Cor. 1.5]{Con03}). We denote them by $J[\fm]_{S''}$ and $\cQ[\fm]_{S''}$ respectively.

Let $\dim_{\F_\ell} \cQ[\fm] = 1+k$, so $\cQ[\fm]$ is an extension of $\mu_\ell^{\oplus k}$ by $\zmod \ell$. 
For $1\leq i\leq k$, let $\iota_i : \mu_\ell \inj \mu_\ell^{\oplus k}$ be the embedding into the $i$-th component and let $E_i$ denote ``the pullback by $\iota_i$,'' which is 
the extension defined by the following diagram:
\begin{equation*}
\xyv{1.5}
\xyh{2}
\xymatrix{
0 \ar[r] & \zmod \ell \ar@{=}[d]\ar[r] & \cQ[\fm]_{S''} \ar[r] & \mu_\ell^{\oplus k} \ar[r] & 0 \\
0 \ar[r] & \zmod \ell \ar[r] & E_i \ar@{^(->}[u] \ar[r] & \mu_\ell \ar[r] \ar@{^(->}[u]_-{\iota_i} & 0.
}
\end{equation*}
Then, $E_i$'s are extensions of $\mu_\ell$ by $\zmod \ell$ over $S''$. 

Now, we use the same convention as Section \ref{section : the group of extensions}. 
Let $L_0$ be the maximal elementary abelian $\ell$-extension of $F$ such that $L/F$ is unramified outside $N_0$ and split over $\l_F$. Let $G_0:=\Gal(L_0/F)$ and $\D:=\Gal(F/\Q)$. 
Let $\Phi(E_i)\in \Hom_{\F_\ell} (G_0, \F_\ell(-1))^{\D}$, defined in (\ref{equation : Phi}).
If $k > \#\cS_\fm$, then the set $\{ \Phi(E_i) : 1 \leq i \leq k \}$ is linearly dependent because
\begin{equation*}
\dim_{\F_\ell} \Ext_{\Z[1/{N_0}]}(\mu_\ell, \zmod \ell)=\dim_{\F_\ell} \Hom_{\F_\ell} (G_0, \F_\ell(-1))^{\D}=\#\cS_\fm.
\end{equation*}
Therefore there exists 
$(a_1, \dots, a_k) \in \F_\ell^k \sm \{(0, \dots, 0)\}$ such that $\sum_{i=1}^k a_i \cdot \Phi(E_i)=0$.  
Then, the map $\iota:= \sum_{i=1}^k a_i \cdot \iota_i : \mu_\ell \to \mu_\ell^{\oplus k}$ is injective and  
``the pullback by $\iota$'' gives rise to an embedding $\zmod \ell \oplus \mu_\ell \inj \cQ[\fm]_{S''}$:
\begin{equation*}
\xyv{1.5}
\xyh{2}
\xymatrix{
0 \ar[r] & \zmod \ell \ar@{=}[d]\ar[r] & \cQ[\fm]_{S''} \ar[r] & \mu_\ell^{\oplus k} \ar[r] & 0 \\
0 \ar[r] & \zmod \ell \ar[r] & \zmod \ell \oplus \mu_\ell \ar@{^(->}[u] \ar[r] & \mu_\ell \ar[r] \ar@{^(->}[u]_-{\iota} & 0.
}
\end{equation*}
This is a contradiction because $\mu_\ell \not\subset \cQ[\fm]$. Therefore $k \leq \# \cS_\fm$ and $\dim_{\F_\ell} \cQ[\fm]=1+k \leq 1 +\# \cS_\fm$ as claimed. 
\end{proof}

\begin{remark}
Let $\cQ[\fm]$ be an extension of $\mu_\ell^{\oplus k}$ by $\zmod \ell$ as above. Let $\Q(\cQ[\fm])$ be the field defined by the points of $\cQ[\fm](\ov{\Q})$. Then, $\Q(\cQ[\fm])$ contains $F$. Let $G^\fm=\Gal(\Q(\cQ[\fm])/\Q)$ and $G_0^\fm=\Gal(\Q(\cQ[\fm])/F)$.
By a suitable choice of basis, we can obtain a Galois representation of the form
\begin{equation*}
\rho_E : G^\fm \ra \left\{\left(\begin{matrix}
    \mathbf{1} & c_1 & c_2 & \dots  & c_k \\
    0 & \chi_\ell & 0 & \dots  & 0 \\
    0&   0 & \chi_\ell & \dots  & 0 \\
    \vdots & \vdots & \vdots & \ddots & \vdots \\
    0 & 0 & 0 & \dots  & \chi_\ell
\end{matrix} \right) : c_i \in H^1(G^\fm, \F_\ell(-1)) \right\}.
\end{equation*}
Then $\Phi(E_i)$ defined above is equal to the composition of the canonical projection $G_0 \to G_0^\fm$ 
and the restriction of $c_i$ to $G_0^\fm$.
\end{remark}

\begin{theorem}\label{theorem : lower bound of dimension}
For each $p \in \cS_\fm$, suppose that there is an injection defined over $\Q$  
\begin{equation*}
i_p : E_p \inj \cQ[\fm],
\end{equation*}
where $E_p \sim \cE_p$. Then, we have
\begin{equation*}
\dim_{\F_\ell} \cQ[\fm]=1+\# \cS_\fm.
\end{equation*}
\end{theorem}
\begin{proof}
By Theorem \ref{theorem : upper bound of dimension}, it suffices to show that $\dim_{\F_\ell} \cQ[\fm] \geq 1+\# \cS_\fm$.

If $\# \cS_\fm=1$ then$\dim_{\F_\ell} \cQ[\fm] \geq 2$ because $\dim_{\F_\ell} E_p=2$.

Now suppose that $\cS_\fm = \{p, q \}$ and we show that $\dim_{\F_\ell} \cQ[\fm] \geq 3$. To do this, it suffices to show that $i_p(E_p) \neq i_q(E_q)$, which is indeed true because the Galois modules associated to $E_p$ (resp. $E_q$) is ramified precisely at $p$ (resp. $q$), and both $i_p$ and $i_q$ are defined over $\Q$. Applying the same argument inductively, we get $\dim_{\F_\ell} \cQ[\fm] \geq 1+\# \cS_\fm$.
\end{proof}

\begin{remark} \label{remark : when dim=1+Sm}
If $\dim_{\F_\ell} \cQ[\fm]=1+\# \cS_\fm$, then there is an injection (defined over $\Q$)
$i_p : \cE_p \inj \cQ[\fm]$ for any $p \in \cS_\fm$. To show this, let $k=\#\cS_\fm$ and let $E_i$ be the extension of $\mu_\ell$ by $\zmod \ell$ defined in the proof of Theorem \ref{theorem : upper bound of dimension}. 
Then, $\Phi(E_i)$'s are linearly independent because $\mu_\ell \not\subset \cQ[\fm]$. Thus, the set $\{ \Phi(E_i) : 1 \leq i \leq k \}$ forms a basis of $\Hom_{\F_\ell} (G_0, \F_\ell(-1))^{\D}$. Since $\Phi(\cE_p) \in \Hom_{\F_\ell} (G_0, \F_\ell(-1))^{\D}$, we have
\begin{equation*}
\Phi(\cE_p) =\sum_{i=1}^k a_i \cdot \Phi(E_i)
\end{equation*}
for some $(a_1, \dots, a_k) \in \F_\ell^k$. 
Since $\Phi(\cE_p)$ is a non-trivial extension, $(a_1, \dots, a_k) \neq (0, \dots, 0)$.
As above, the pullback by $\sum_{i=1}^k a_i \cdot \iota_i : \mu_\ell \to \mu_\ell^{\oplus k}$, which is an embedding, defines an injection $i_p : \cE_p \inj \cQ[\fm]$. Since this embedding is compatible with Galois action, it is defined over $\Q$.
\end{remark}

\ms 

\section{Proofs} \label{section : proof}
In this section, we prove Theorems \ref{theorem : the main theorem}, \ref{theorem : level qr - level qr2} and \ref{theorem : main theorem 2}. Before doing that we introduce a ``level raising method'' which is very useful to check the assumption in Theorem \ref{theorem : lower bound of dimension}.
\subsection{Level raising method}
Let $\fm$ be a new rational Eisenstein prime of $\T(N)$ containing $\ell$. 
In this subsection, we assume that $\ell$ does not divide $2N$.
For any divisor $M$ of $N$, we define the ideal $\fm(M)$ of the Hecke ring of level $M$ corresponding to $\fm$ as follows:
\begin{equation*}
\fm(M):=(\ell, T_p-\e(p) : \text{ for all primes } p),
\end{equation*}
where 
\begin{equation*}
\e(p):=\begin{cases}
p+1 & \text{ if $p$ does not divide $M$},\\
~~0 & \text{ if $p^2$ divides $M$},\\
-1 & \text{ if  $p$ divides $M$ and } T_p \equiv -1 \pmod \fm, \\
~~1 & \text{  otherwise}.
\end{cases}
\end{equation*}
Since $\fm$ is new, if $T_p \equiv -1 \pmod \fm$, then $p^2$ does not divide $N$. Therefore $p^2$ does not divide $M$, either.
Note that $\fm(M)$ may not be maximal.

\begin{proposition}\label{proposition : level raising}
Let $N=Mp$ and let $\fm$ be a rational Eisenstein prime of $\T(N)$, which is new.
Suppose that $\fm(M)$ is maximal. Let $\cQ[\fm(M)]$ denote the quotient of $J_0(M)[\fm(M)]$ by $\S(M)[\fm(M)]$. Then, there is an injection defined over $\Q$
\begin{equation*}
\iota : \cQ[\fm(M)] \inj \cQ[\fm].
\end{equation*}
\end{proposition}
\begin{proof}
For a prime divisor $p$ of $N$, let $a_p \in \{0, 1, -1\}$ denote the image of $T_p$ in $\T(N)/\fm$. Since we assume that $\fm$ is new, $a_p = \pm 1$ if $p$ does not divide $M$, and $a_p=0$ otherwise. 

For ease of notation, we set $\fn:=\fm(M)$. Moreover, we set
\begin{equation*}
J:=J_0(M), J':=J_0(N), \S:=\S(M) \text{ and }\S':=\S(N).
\end{equation*}

In the four cases below, we will define the morphism $\g$ defined over $\Q$
\begin{equation*}
\g : J \to J'
\end{equation*}
so that the restriction of $\g$ on $J[\fn]$ induces a map $J[\fn] \to J'[\fm]$, which we also denote by $\g$. Furthermore, it gives rise to a commutative diagram:
\begin{equation*}
\xymatrix{
0 \ar[r] & \S[\fn] \ar[d]_-\g \ar[r] & J[\fn] \ar[d]_-\g  \ar[r] \ar[dr]^-{\pi \circ \g}&\cQ[\fn] \ar@{-->}[d]^-\iota \ar[r] & 0\\
0 \ar[r] & \S'[\fm] \ar[r] & J'[\fm]  \ar[r]^-\pi &\cQ[\fm] \ar[r] & 0.
}
\end{equation*}
Then, we will show that the kernel of $\pi \circ \g$ is $\S[\fn]$, so the dotted arrow in the commutative diagram is injective. Since all the maps in the diagram except $\iota$ are defined over $\Q$, so is $\iota$. Therefore the result follows.

Now, let $\a_p^*, \b_p^* : J \to J'$ denote two degeneracy maps in Section \ref{section : background}, which are defined over $\Q$. 
\begin{itemize}
\item Case 1: $(p, M)=1$ and $a_p=1$. For $x\in J$, let $\g(x):=\a_p^*(x)-\b_p^*(x)$. 

\item Case 2: $(p, M)=1$ and $a_p=-1$. For $x\in J$, let $\g(x):=p\a_p^*(x)-\b_p^*(x)$. Note that $p \equiv -1 \pmod \ell$ because $\fm$ is new.

\item Case 3: $p \mid M$ but $p^2 \nmid M$. For $x\in J$, let $\g(x):=p\a_p^*(x)-\b_p^*(x)$. (By definition, $\e(p)=1$.) 

\item Case 4: $p^2 \mid M$. For $x\in J$, let $\g(x):=\a_p^*(x)$. (By definition, $\e(p)=0$.)
\end{itemize}
In all cases, $\g(J[\fn]) \subset J'[\fm]$ by \cite[Prop.  2.4]{Yoo6}.

Let $K$ denote the kernel of $\pi \circ \g$. By \cite[Th. 4]{LO91} $\a_p^*(\S) \subset \S'$ and $\b_p^*(\S) \subset \S'$. Thus, $\g(\S) \subset \S'$ and hence
\begin{equation*}
\g(\S[\fn])=\g(\S \cap J[\fn]) \subset \g(\S) \cap \g(J[\fn])\subset \S' \cap J'[\fm]=\S'[\fm].
\end{equation*}
Therefore $\pi\circ \g(\S[\fn])=0$ and $\S[\fn] \subset K$.

On the other hand, in all four cases above, 
the intersection of the kernel of $\g$ and $\zmod \ell \subset J[\fn]$ is trivial by \cite[Th. 1.4]{Yoo6}.  Since $\S[\fn] \simeq \mu_\ell^{\oplus a}$ for some $a$, the intersection of $K$ and $\zmod \ell$ is trivial as well. This implies that $K$ is of multiplicative type because $J[\fn]^\text{ss} \simeq \zmod \ell \oplus \mu_\ell^{\oplus n}$ for some $n\geq 1$.
Since $K$ is annihilated by $\fn$, by the same argument as in Proposition \ref{proposition : criterion of constancy}, $K \simeq \mu_\ell^{\oplus c}$ for some $c$. By Proposition \ref{proposition : vatsal}, $K=K[\fn] \subset \S \cap J[\fn]=\S[\fn]$. Therefore $K=\S[\fn]$ and $\iota$ is injective.
\end{proof}

Applying this proposition inductively, we easily get the following:
\begin{corollary}\label{corollary : level raising}
Let $M$ be a divisor of $N$, and let $\fm$ be a rational Eisenstein prime of $\T(N)$, which is new. If $\fm(M)$ is still maximal, then there is an injection defined over $\Q$
\begin{equation*}
\iota : \cQ[\fm(M)] \inj \cQ[\fm].
\end{equation*}
\end{corollary}
\begin{remark}
In the proof of Proposition \ref{proposition : level raising}, the assumption that $\fm$ is new is only used to get 
$a_p \in \{ 0, 1, -1 \}$, and $p \equiv -1 \pmod \ell$ if $a_p=-1$. Therefore when $\fm:=\fm_\ell(s, t, u)$ (without assuming $\fm$ is new), we can still get the same result.
\end{remark}
\ms

\subsection{Proof of Theorem \ref{theorem : the main theorem}}
From now on, let $\ell \geq 5$ be a prime and let 
\begin{equation*}
N=\prod_{i=1}^s p_i \prod_{j=1}^t q_j \prod_{k=1}^u r_k^{e(k)} ~~\text{ with } e(k)\geq 2
\end{equation*}
not divisible by $\ell$.
Let $\fm:=\fm_\ell(s, t, u)$ be a rational Eisenstein prime of $\T(N)$. 
By Assumption \ref{assumption : m=m(s, t, u)}, we have
\begin{equation*}
\cS_\fm = \{ p_1, \dots, p_{s_0}, q_1, \dots, q_t, r_1, \dots, r_{u_1} \}
\end{equation*}
and $\#\cS_\fm = s_0+t+u_1$, which is at least $1$ by Lemma \ref{lemma : maximal ideal assumption}. 
Note also that $J_0[\fm]=\cQ[\fm]$ by Proposition  \ref{proposition : Ling-Oesterle}
since $N$ is not squarefree.
Therefore, when $s_0\neq s$, or $u_0 \neq u$, or $t=0$, it suffices to check that
the assumption in Theorem \ref{theorem : lower bound of dimension} is fulfilled. In other words, it is enough to show that for each prime $q \in \cS_\fm$, there is an $E_q$ such that $E_q\sim \cE_q$ and $E_q \inj \cQ[\fm]$, which is defined over $\Q$.
\ms

{\bf Case 1 : $\bd{s_0 \neq s}$ or $\bd{u_0 \neq u}$.} 

Suppose that $q=r_k$ for some $1\leq k \leq u_1$. Let $\fn=\fm(q^2)$. Then, we claim that we can take $E_q=\cQ[\fn]$.
By Corollary \ref{corollary : level raising} we have $E_q \inj \cQ[\fm]$, defined over $\Q$. So, it suffices to show that $\cQ[\fn] \sim \cE_q$, which is obtained by the following:
\begin{itemize}
\item
By Proposition \ref{proposition : Ling-Oesterle}, $\S(q^2)[\fn]=0$ because $T_p \equiv 0 \not\equiv p \pmod \fm$. Therefore $\cQ[\fn]=J_0(q^2)[\fn]$.
\item
Note that $\cS_\fn=\{ q\}$. Therefore, $\dim_{\F_\ell} \cQ[\fn]=2$ 
by Lemma \ref{lemma : dim at least 2} and Theorem \ref{theorem : upper bound of dimension}.
\item
By Corollary \ref{corollary : structure of cQ[m]}, $\cQ[\fn] \in \Ext_{\Z[1/q]}(\mu_\ell, \zmod \ell)$, which is non-trivial.
\end{itemize}

Now, suppose that $q\in \cS_\fm \sm \{ r_1, \dots, r_{u_1} \}$, i.e., $q=p_i$  for some $1\leq i\leq s_0$ or $q=q_j$ for some $1\leq j\leq t$. 
Let $p:=p_s$ if $s_0 \neq s$, and let $p:=r_u$ if $s_0=s$ but $u_0 \neq u$. 
Then by assumption $p \not\equiv 1 \pmod \ell$. 
Let $\fn=\fm(pq)$, which is maximal by Lemma \ref{lemma : maximal ideal assumption}.
Then, we claim that we can take $E_q=\cQ[\fn]$.
By Corollary \ref{corollary : level raising}, it suffices to show that $\cQ[\fn] \sim \cE_q$. This follows from:
\begin{itemize}
\item
$\dim_{\F_\ell} J_0(pq)[\fn]=2$ by \cite[Th. 1.3]{Yoo1}.
\item
By Proposition \ref{proposition : Ling-Oesterle}, $\S(pq)[\fn]=0$ because $T_p \equiv 1 \not\equiv p \pmod \fm$. Thus, $\cQ[\fn]=J_0(pq)[\fn]$.
\item
By Corollary \ref{corollary : structure of cQ[m]}, $\cQ[\fn]$ is a non-trivial extension of $\mu_\ell$ by $\zmod \ell$. 
\item
By Corollary \ref{corollary : unramified of Q[m]}, $\cQ[\fn]$ is unramified at $p$ and hence 
$\cQ[\fn] \in \Ext_{\Z[1/q]}(\mu_\ell, \zmod \ell)$. 
\end{itemize}
\ms

{\bf Case 2 : $\bd{t=0}$}

If either $s_0\neq s$ or $u_0 \neq u$, then the claim follows from the cases above. Thus, we assume that $s_0=s$ and $u_0=u$, i.e, all prime divisors $p$ of $N$ are congruent to $1$ modulo $\ell$. Let $q \in \cS_\fm$. 

If $q=r_k$ for some $1\leq k\leq u$, then we can take $E_q=\cQ[\fm(q^2)]$ as in Case 1 above. 

Now let $q=p_i$ for some $1\leq i \leq s$. 
To find an appropriate $E_q$, we apply some ideas used in the paper \cite{RY}. 
For ease of notation, we set $\fa:=\fm(q)$ and $\T:=\T(q)$. Moreover, we set
\begin{equation*}
J:=J_0(q) \text{ and } \S:=\S(q).
\end{equation*}
Let $\S_\ell$ denote the $\ell$-primary part of $\S$ and let
\begin{equation*}
A:=J/{\S_\ell} \qa \pi : J \to A
\end{equation*}
be the corresponding quotient isogeny. Note that $\T$ acts on $A$ via $\pi$ because $\S_\ell$ is stable under the action of $\T$. We claim that $A[\fa] \sim \cE_q$ and there is an injection defined over $\Q$
\begin{equation*}
A[\fa] \inj \cQ[\fm(qr_1)],
\end{equation*}
which induces an injection $A[\fa] \inj \cQ[\fm]$ by the composition with the map
$\cQ[\fm(qr_1)] \inj \cQ[\fm]$ given in Corollary \ref{corollary : level raising}. (By our assumption, $u_0=u\geq 1$ so we can pick $r_1 \equiv 1 \pmod \ell$.)
Therefore we can take $E_q=A[\fa]$.

In the below, we prove our claim as following steps:
\begin{itemize}
\item
Step 1: $\dim_{\F_\ell} A[\fa] \geq 2$.
\item
Step 2: Let $V:=\pi^{-1}(A[\fa]) \subset J$.
Then, $V$ is annihilated by $\eta_p:=T_p-p-1$, for all bad primes $p$.
\item
Step 3: Fix a bad prime $p$ (which is defined below). Let $\g : J \to J_0(pq)$ be the map defined by $\g(x):=\a_p^*(x)-\b_p^*(x)$ (as in the proof of Proposition \ref{proposition : level raising}). Then, $B_p:=\g(V)$ is isomorphic to $A[\fa]$ and is annihilated by 
\begin{equation*}
\fb:=(\ell, U_p-1, T_q-1, \eta_r : \text{ for primes } r \nmid pq) \subset \T(pq).
\end{equation*}
(Here, we denote the $p$-th Hecke operator of level $pq$ by $U_p$ to distinguish it from the $p$-th Hecke operator of level $q$, which we already denoted by $T_p$.) 
\item
Step 4: For appropriately chosen $p$, $A[\fa] \simeq B_p=J_0(pq)[\fb] \sim \cE_q$.
\item
Step 5: By taking $p=r_1$, we get $A[\fa] \simeq B_{r_1} \inj \cQ[\fm(qr_1)]$, as claimed.
\end{itemize}

For Step 1 we apply Mazur's argument used in Lemma \ref{lemma : dim at least 2}. Since $A$ is isogenous to $J$, $\Ta_{\fa}(A)\otimes_\Z \Q$ is also of dimension $2$ over $\T_\fa \otimes_\Z \Q$, which implies that $\dim_{\F_\ell} A[\fa] \geq 2$.

For Step 2, we recall some results of Mazur in \cite{M77}. Let
\begin{equation*}
\cI = \cI_0(q):=(\eta_r : \text{ for all primes } r \neq q ) 
\end{equation*}
be the Eisenstein ideal of prime level $q$, where $\eta_r:=T_r-r-1$.
Then, $\cI \T_{\fa} \subset \T_\fa$ is principal and  $\eta_r$ is a generator for $\cI \T_\fa$ if and only if $r$ is a good prime (relative to the pair $(\ell, q)$) (Theorem 18.10 of Chapter II in \textit{loc. cit.}). Here, a prime $r$ is {\sf good} if both $r \not\equiv 1 \pmod \ell$ and $r$ is not an $\ell$-th power modulo $q$ (page 124 of \textit{loc. cit.}). We say a prime is {\sf bad} if it is not good. 
Choose a good prime $M$, so let $\cI \T_\fa = (\eta_M)$. 
Let $p$ be a bad prime. Then, $\eta_p$ is not a generator for $\cI \T_\fa$ so, $\eta_p = \l \eta_M$ for some $\l \in \fa\T_\fa$.\footnote{If $\l \not\in \fa\T_\fa$, then $\eta_p$ can generate $\cI \T_\fa$, which is a contradiction.} 
Let $A_\fa$ denote the $\fa$-divisible group associated to $A$ as in Section \ref{section : Mazur's argument}. Namely, $A_\fa :=\ilims k A[\fa^k]$ on which $\T_\fa$ acts. Via the canonical injection $A[\fa] \inj A_\fa$,
$\l \in \fa\T_\fa$ annihilates $A[\fa]$. Now let $\ell^k$ be the exact power of $\ell$ dividing $q-1$. Then, $\S_\ell$ is of order $\ell^k$ and $V \subset J[\ell^{k+1}]$. Furthermore, since $\S_\ell$ is killed by $\fa^k$, $V$ is killed by $\fa^{k+1}$. 
Thus, the action of $\T_\ell:=\T \otimes_\Z \Z_\ell$ on $V$ (viewed as a submodule in the $\ell$-divisible module $J_\ell$) factors through that of $\T_\fa$. Note that $\l(V) \subset \S_\ell$ because $\pi(\l (V)) \subset \l(\pi(V))=\l(A[\fa])=0$.
So, we have
\begin{equation*}
\eta_p(V)=\l \eta_M(V)=\eta_M( \l (V)) \subset \eta_M(\S_\ell)=0,
\end{equation*}
which implies $\eta_p(V)=0$ as claimed.

For Step 3, we note that the kernel of $\g$ is $\S$ by Ribet \cite[Th. 4.3]{R83}. Since $\S \cap V = \S_\ell$, we have $A[\fa] \simeq B_p$.
Now we show that $B_p$ is annihilated by $\fb$. Let $T \in \fb$ be a generator different from $U_p-1$ (i.e., $T=\ell$, $T_q-1$, or $\eta_r$ for primes $r\nmid pq$). 
By abuse of notation, $T$ can be regarded as an element of $\T$ and we get $T \in \fa$.
Therefore for any $x \in V$ we have $Tx \in \S_\ell$ because $\pi(Tx)=T(\pi x)=0$. (Note that $\pi$ commutes with the action of $\T$.) Since $\g$ commutes with $T$ and its kernel is $\S$, we get $T(\g(x))=\g(Tx)=0$ for any $x \in V$. In other words, $T$ annihilates $B_p$. 
Now let $T=U_p-1$. 
By Formula (\ref{equation : Up and Tp}) in Section \ref{section : background} and the fact that 
$V$ is annihilated by $\eta_p$ (Step 2), 
for any $x \in V$
\begin{equation*}
T(\g(x))=(U_p-1)(\g(x))=\g_p \mat {T_p-1} p {-1} {-1}\vect x {-x} =0.
\end{equation*} 
Therefore $B_p$ is annihilated by $\fb$ as insisted.

For Step 4, choose a prime $p$ such that $p\not\equiv \pm 1 \pmod \ell$ but $p\equiv 1 \pmod q$. (By the Dirichlet's theorem on arithmetic progression, we can do that.) Since $1$ is an $\ell$-th power modulo $q$, $p$ is a bad prime.
By \cite[Th. 4.9(1)]{Yoo1} and Proposition \ref{proposition : Ling-Oesterle}, $\dim_{\F_\ell} J_0(pq)[\fb]=\dim_{\F_\ell} \cQ[\fm]=2$. Since $\cS_\fb=\{ q \}$, we get $\cE_q \sim J_0(pq)[\fb]$. Since $A[\fa] \simeq B_p \subset J_0(pq)[\fb]$ by Step 3 and $\dim_{\F_\ell} A[\fa] \geq 2$ by Step 1, $A[\fa] \simeq B_p = J_0(pq)[\fb] \sim \cE_q$.

Finally, for Step 5 we take $p=r_1$, which is a bad prime. Then, $\fb=\fm(qr_1)$ and 
$A[\fa] \simeq B_{r_1} \subset J_0(qr_1)[\fb]$. Since $B_{r_1} \sim \cE_q$, $B_{r_1}$ cannot contain a module isomorphic to $\mu_\ell$. Therefore 
$B_{r_1} \cap \S(qr_1)=0$ and the following composition map is injective:
\begin{equation*}
A[\fa] \simeq B_{r_1} \inj J_0(qr_1)[\fb] \surj \cQ[\fb].
\end{equation*}
In both cases above, we have proved that there is an injection $E_q \inj \cQ[\fm]$ defined over $\Q$ for any $q \in \cS_\fm$, so the result follows from Theorem \ref{theorem : lower bound of dimension}.
\qed

\ms
\subsection{Proof of Theorem \ref{theorem : level qr - level qr2}}
By Theorem \ref{theorem : lower bound of dimension}, it suffices to find $E_q$ and $E_r$ satisfying the assumption.
Note that we can take $E_r=\cQ[\fm(r^2)]$ as in the proof of Theorem \ref{theorem : the main theorem} without any other hypothesis.

Now, we claim that we can take $E_q=\cQ[\fn]$ if $\dim_{\F_\ell} J_0(qr)[\fn]=3$.
Suppose  $\dim_{\F_\ell} J_0(qr)[\fn]=3$. 
First of all, $J_0(qr)[\fn]$ is unramified at $r$ by \cite[Th. 4.5]{Yoo1}. Thus, $\cQ[\fn]$ is also unramified at $r$.
Then, since $\S(qr)[\fn] \simeq \mu_\ell$ by Proposition \ref{proposition : Ling-Oesterle}, $\dim_{\F_\ell} \cQ[\fn]=2$ and hence $\cQ[\fn]$ is a non-trivial extension of $\mu_\ell$ by $\zmod \ell$ over $\Z[1/q]$. This implies that $\cQ[\fn] \sim \cE_q$. Finally, since $\fn=\fm(qr)$, by Proposition \ref{proposition : level raising} we get an injection $\cQ[\fn] \inj \cQ[\fm]$ defined over $\Q$, as claimed.
\qed

\ms
\subsection{Proof of Theorem \ref{theorem : main theorem 2}}
Suppose that $s_0=s$, $u_0=u\geq 1$ and $t\geq 1$.
Let $\fm:=\fm_\ell(s, t, u)$ be a rational Eisenstein prime of $\T(N)$. Then, 
\begin{equation*}
\cS_\fm = \cS(N) = \{ p_1, \dots, p_s, q_1, \dots, q_t, r_1, \dots, r_u \}
\end{equation*}
and $\# \cS_\fm = s+t+u$. 

First, $J_0(N)[\fm]=\cQ[\fm]$ by Proposition \ref{proposition : Ling-Oesterle}. Therefore by Theorem \ref{theorem : lower bound of dimension},
it suffices to show that there is an injection $E_q \inj \cQ[\fm]$ defined over $\Q$ for all $q \in \cS_\fm$, where $E_q \sim \cE_q$.

Then, for $q=p_i$ or $q=r_k$, we can find an $E_q$ as in Case 1 of the proof of Theorem \ref{theorem : the main theorem} above. So, it suffices to show that $\cE_q \inj \cQ[\fm]$ for $q=q_j$.

Last, let $q=q_j$ and $r=r_1$. Suppose Conjecture \ref{conjecture : t=1, u=1} holds. Then, 
$\dim_{\F_\ell} J_0(qr^2)[\fm(qr^2)]=3$.
Since $J_0(qr^2)[\fm(qr^2)]=\cQ[\fm(qr^2)]$ by Proposition \ref{proposition : Ling-Oesterle}, there is an injection $\cE_q \inj \cQ[\fm(qr^2)]$ defined over $\Q$ by Remark \ref{remark : when dim=1+Sm}. Composing with the injection $\cQ[\fm(qr^2)] \inj \cQ[\fm]$ in Corollary \ref{corollary : level raising}, we get a desired injection
$\cE_q \inj \cQ[\fm]$ defined over $\Q$ and the theorem follows.
\qed

\end{document}